\documentclass[12pt]{article}

\usepackage[margin=1in]{geometry}
\usepackage{amsthm}
\usepackage{amssymb}
\usepackage{amsmath}
\usepackage{graphicx}
\usepackage{url}
\usepackage{color}
\usepackage{framed}
\usepackage{hyperref}

\newtheorem{theorem}{Theorem}

\newtheorem{lemma}[theorem]{Lemma}
\newtheorem{proposition}[theorem]{Proposition}
\newtheorem{conjecture}[theorem]{Conjecture}
\newtheorem{conjx}{Conjecture}

\newcommand{\ejk}[1]{#1}

\newcommand{\abs}[1]{\left\vert #1 \right\vert}
\newcommand{\eps}{\epsilon}

\theoremstyle{definition}

\newtheorem{example}[theorem]{Example}
\newtheorem{definition}[theorem]{Definition}

\title{Uniquely optimal codes of low complexity are symmetric}

\author{Emily~J.~King\footnote{Department of Mathematics, Colorado State University, Fort Collins, CO, USA} \quad
Dustin~G.~Mixon\footnote{Department of Mathematics, The Ohio State University, Columbus, OH, USA} \quad Hans~Parshall\footnote{WellSense Health Plan, Charlestown, MA, USA}
\quad Chris~Wells\footnote{Department of Mathematics and Statistics, Auburn University, Auburn, AL, USA}}
\date{}

\begin{document}
\maketitle

\begin{abstract}
We formulate explicit predictions concerning the symmetry of optimal codes in compact metric spaces.
This motivates the study of optimal codes in various spaces where these predictions can be tested.
\end{abstract}
\section{Introduction}

Solutions to geometric extremal problems often exhibit a notably high degree of symmetry.
Fejes T\'{o}th observed this phenomenon in~\cite{Toth:64,Toth:86}, in which he elaborates on many examples, including the Tammes problem of arranging points on the sphere so that the minimum distance is maximized.
For this problem, optimal configurations include the vertices of the tetrahedron, the octahedron, and the icosahedron~\cite{Toth:40}.
By virtue of their striking symmetry, these Platonic solids were well understood by Euclid long before the Dutch botanist Tammes was inspired by the regular distribution of pores on spherical pollen grains, and yet they independently arise as solutions to a seemingly unrelated geometric extremal problem.

The recent literature offers numerous incarnations of this mysterious correspondence between optimality and symmetry.
Cohn and Kumar~\cite{CohnK:07} showed that there are a handful of configurations in $S^{d-1}$ that simultaneously minimize an infinite class of natural choices of potentials, and each of these configurations curiously exhibits a high degree of symmetry.
Viazovska~\cite{Viazovska:17} established that the $E_8$ lattice gives the densest possible packing of spheres in $\mathbb{R}^8$, and a similar approach later produced the analogous result for $\mathbb{R}^{24}$ in terms of the Leech lattice~\cite{CohnKMRV:17}.
De Grey~\cite{deGrey:18} identified the first known $5$-chromatic unit-distance graph, and a follow-on Polymath project~\cite{Polymath:online} established that the corresponding planar configuration of $1581$ points resides in a ring extension of $\mathbb{Z}$ generated by only four complex numbers.
Kopp~\cite{Kopp:18} provided evidence that arrangements of $d^2$ points in $\mathbb{CP}^{d-1}$ that maximize the minimum distance are given by the Heisenberg--Weyl orbit of a point that is easily expressed in terms of Stark units. This connection has been strengthened in further articles, e.g., \cite{ApplebyFK25}.

These baffling coincidences between optimality and symmetry demand an explanation.
Along these lines, the authors are aware of two partial explanations in two specific settings.
First, the Erd\H{o}s distinct distances problem asks for the asymptotic form of the minimum number $d(n)$ of distinct distances between $n$ points in the plane.
Elekes and Sharir~\cite{ElekesS:10} showed that a set with few distances necessarily exhibits a large number of partial symmetries, and Guth and Katz~\cite{GuthK:15} later managed to bound these partial symmetries to obtain the near-optimal lower bound of $d(n)=\Omega(n/\log n)$.
As such, in this setting, it is understood how optimality implies (partial) symmetry.
For the reverse direction, we look to the setting of optimal codes in complex projective spaces.
Here, we seek arrangements of $n$ points in $\mathbb{CP}^{d-1}$ that maximize the minimum distance.
Consider an arrangement with the property that for every pair of ordered pairs of distinct points, there exists a projective unitary operator that permutes the arrangement while mapping one ordered pair to the other.
As established in~\cite{IversonM:18,IversonM:19}, these doubly transitive arrangements are necessarily optimal projective codes whenever $n>d$.
That is, in some sense, symmetry implies optimality.

In this paper, we propose an explanation in the ``optimality implies symmetry'' direction that appears to hold in a much more general setting.
In particular, we say an arrangement $C$ of points in a compact metric space $(M,d)$ forms an \textbf{optimal code} if its minimum distance $\delta(C)$ is as large as possible, and we consider compact metric spaces that enjoy a nontrivial isometry group.
In this setting, we conjecture that uniquely optimal codes of low complexity are necessarily invariant under some nontrivial isometry.
We start in the next section by developing some intuition in the context of spherical codes.
Next, Section~3 provides the main definitions and conjectures for the remainder of the paper.
These conjectures elevate some of Fejes T\'{o}th's observations to explicit predictions that can be tested with examples.
In Section~4, we consider a laundry list of such examples that illustrate the veracity of our conjectures.
(For the record, we studied most of these examples \textit{after} formulating our conjectures.)
We conclude in Section~5 with a discussion of various open problems.

\section{A motivating example}

Consider the problem of spherical codes, where for each $d$ and $n$, we seek arrangements of $n$ points in $S^{d-1}$ that maximize the minimum distance.
Any natural choice of distance between $x,y\in S^{d-1}$ is a decreasing function of $\langle x,y\rangle$.
As such, we seek $X\subseteq S^{d-1}$ of size $n$ for which $\max\{\langle x,y\rangle:x,y\in X,x\neq y\}$ is minimized.
Rankin~\cite{DavenportH:51,Rankin:55} provides two relevant bounds:

\begin{proposition}
\label{prop.rankin}
Consider any $X\subseteq S^{d-1}$ of size $n$.
Then
\begin{itemize}
\item[(a)]
$\max\{\langle x,y\rangle:x,y\in X,x\neq y\}\geq-\tfrac{1}{n-1}$, and
\item[(b)]
$\max\{\langle x,y\rangle:x,y\in X,x\neq y\}\geq0$ whenever $n\geq d+2$.
\end{itemize}
\end{proposition}

See~\cite{EricsonZ:01} for a modern treatment of these bounds.
Equality is achieved in Propositon~\ref{prop.rankin}(a) precisely when $X$ forms the vertices of an $(n-1)$-dimensional simplex centered at the origin.
For this reason, Propositon~\ref{prop.rankin}(a) is known as \textbf{Rankin's simplex bound}.
Meanwhile, Propositon~\ref{prop.rankin}(b) is known as \textbf{Rankin's orthoplex bound} since equality is achieved by the vertices of an orthoplex when $n=2d$.
Observe that removing at most $d-2$ points from the orthoplex produces yet another optimal code.

Beyond the spherical codes that achieve equality in Rankin's bounds, there are a handful of examples that have been proven optimal.
Some of these are highly symmetric, such as the icosahedron~\cite{Toth:40} or the snub cube~\cite{Robinson:61}.
See~\cite{CohnK:07,BallingerBCGKS:09} and references therein for additional higher-dimensional examples.
Others are far less symmetric and arise from a tour de force of global optimization; see~\cite{MusinT:15} and references therein.
As one might expect, optimal codes that lack symmetry are much more complicated to describe (i.e., they exhibit large complexity in some subjective sense).
In small dimensions, numerical optimization has delivered various putatively optimal codes that are available in~\cite{Sloane:online}.
To date, there is no $d\geq3$ for which there are infinitely many $n$ such that an explicit code of size $n$ in $S^{d-1}$ is known (or even conjectured) to be optimal.

This paper draws inspiration from codes that achieve equality in Rankin's orthoplex bound.
Before collecting further observations, we provide a characterization of these codes.
(Here and throughout, we denote $[n]:=\{1,\ldots,n\}$.)

\begin{theorem}
\label{thm.orthoplex characterization}
Fix $d\geq 2$ and $k\in\{2,\ldots,d\}$ and consider any $X\subseteq S^{d-1}$ with $|X|=d+k$ that achieves equality in Rankin's orthoplex bound.
Then there exists a possibly empty subset $X_0\subseteq X$ and a partition $X_1\sqcup\cdots\sqcup X_l=X\setminus X_0$ with $l\geq k$ such that
\begin{itemize}
\item[(i)]
$|X_0|=\operatorname{dim}\operatorname{span}X_0$,
\item[(ii)]
$|X_i|=\operatorname{dim}\operatorname{span}X_i+1$ for each $i\in[l]$, and
\item[(iii)]
$\operatorname{span}X_i\perp\operatorname{span}X_j$ whenever $i\neq j$.
\end{itemize}
\end{theorem}
We note that this result also appears as Theorem 3 in \cite{Kuperberg:07}, where the proof is geometric in nature.
\begin{proof}
Let $G$ denote the Gram matrix of $X$, and interpret $A:=I-G$ as a weighted adjacency matrix of a graph. 
Permute rows and columns of $A$ so that it is block diagonal with each block corresponding to a connected component.
For each submatrix that resides in a diagonal block, we apply the Perron--Frobenius theorem for irreducible matrices to conclude that the largest eigenvalue has multiplicity $1$, and furthermore, the corresponding eigenvector is strictly positive on that block.
Considering the top eigenvalues of $A$ correspond to the bottom eigenvalues of $G$, it follows that the null space of $G$ is spanned by some subset $\{v_1,\ldots,v_l\}$ of these eigenvectors.

Notice that $l\geq k$ by rank--nullity.
Next, for each $i\in[l]$ let $X_i \subseteq X$ denote the set of indices of nonzero entries in $v_i$. Note that there is a unique linear dependency on $X_i$ since the corresponding eigenvalue in the corresponding block of $A$ has multiplicity $1$.
This implies (ii).
Put $X_0:=X\setminus\bigcup_{i=1}^l X_i$.
Then (iii) follows from the block structure of $A$.
Finally, we obtain (i):
\[
\operatorname{dim}\operatorname{span}X_0
=\operatorname{dim}\operatorname{span}X-\sum_{i=1}^l\operatorname{dim}\operatorname{span}X_i
=(|X|-l)-\sum_{i=1}^l(|X_i|-1)
=|X_0|.
\qedhere
\]
\end{proof}

Note that Theorem~\ref{thm.orthoplex characterization}(iii) indicates that $k\leq l\leq d$, and so equality in Rankin's orthoplex bound requires $|X|\leq 2d$.
Furthermore, if $k=d$, then $l=d$, in which case (iii) implies that $X_0$ is empty and $\operatorname{dim}\operatorname{span}X_i=1$ for each $i\in[l]$, and then (ii) and (iii) together imply that $X$ consists of the vertices of a $d$-dimensional orthoplex.
Overall, the orthoplex is the largest code that achieves equality in the orthoplex bound, and is uniquely optimal up to isometry.
(This is already well known.)
To illustrate other consequences of Theorem~\ref{thm.orthoplex characterization}, we first take $d=3$, where the only remaining case to consider is $|X|=5$.
In this case, $k=2$, and so (iii) implies that $l\in\{2,3\}$.
However, (ii) implies that $|X|\geq 2l$, and so $|X|=5$ only if $l=2$.
Without loss of generality, we take $|X_1|\leq |X_2|$.
Then $X=X_0\sqcup X_1\sqcup X_2$ takes one of two forms.
If $X_0$ is empty, then $X_1$ consists of two antipodal points (at the north and south poles, say) and $X_2$ consists of three points on the equator.
If $X_0$ is nonempty, then $X$ is obtained by removing a single point form an orthoplex, and $X_0$ contains the antipode of the removed vertex.
See Figure~\ref{fig.orthoplex} for an illustration of such codes.

\begin{figure}
\begin{center}
\includegraphics[width=0.24\textwidth,trim={4cm 7cm 3cm 6cm},clip]{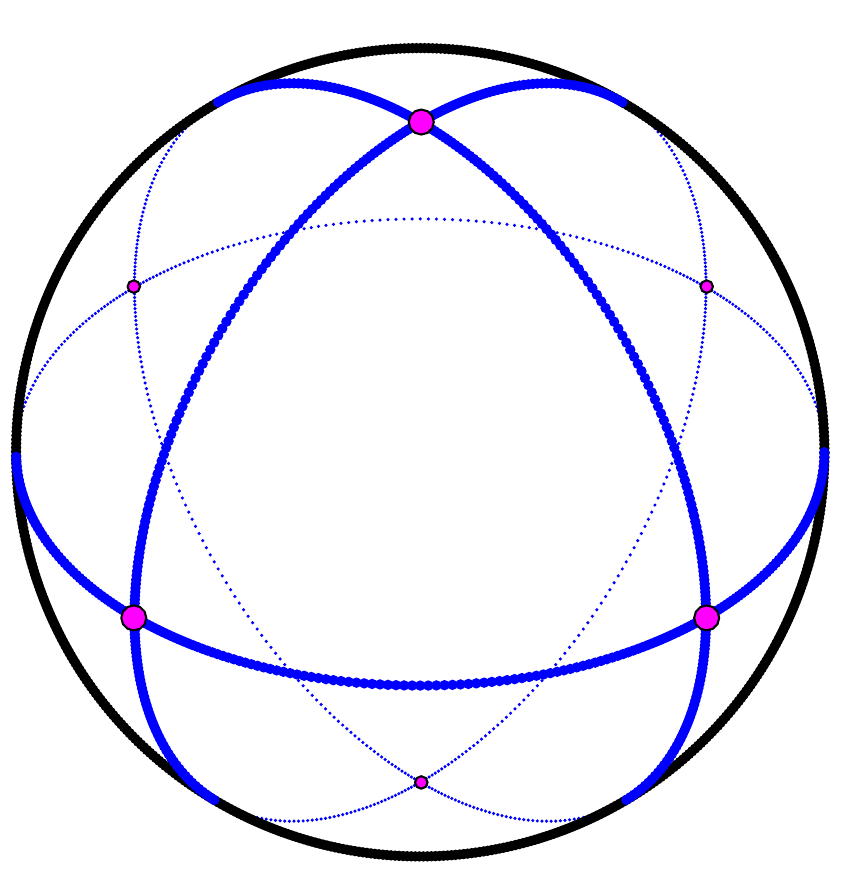}
\includegraphics[width=0.24\textwidth,trim={4cm 7cm 3cm 6cm},clip]{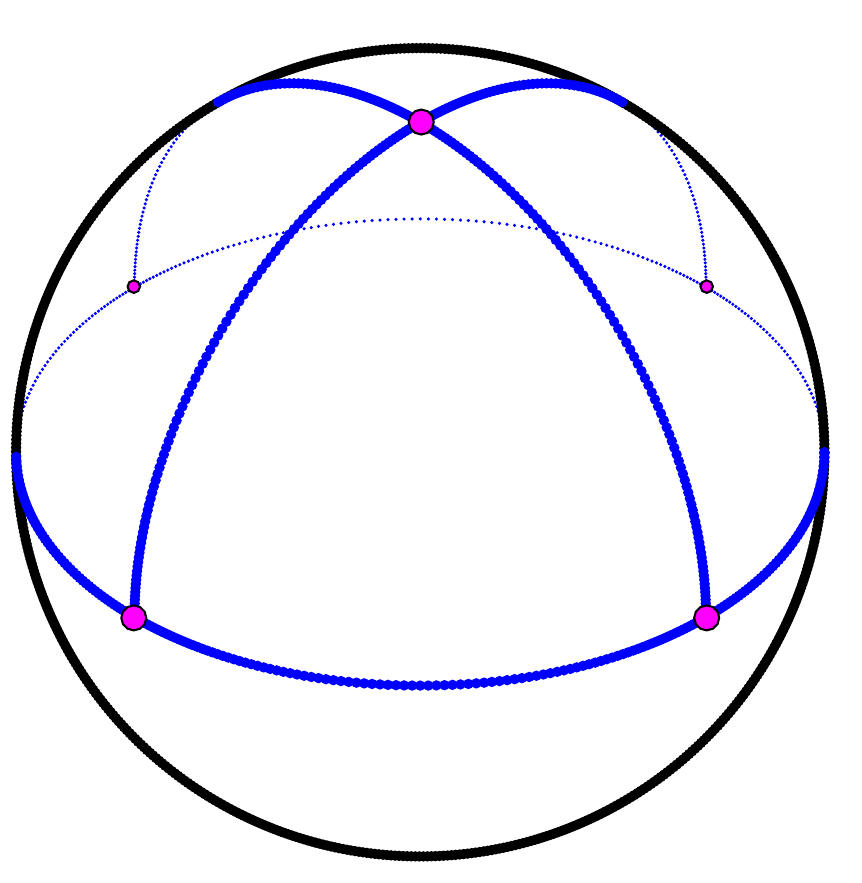}
\includegraphics[width=0.24\textwidth,trim={4cm 7cm 3cm 6cm},clip]{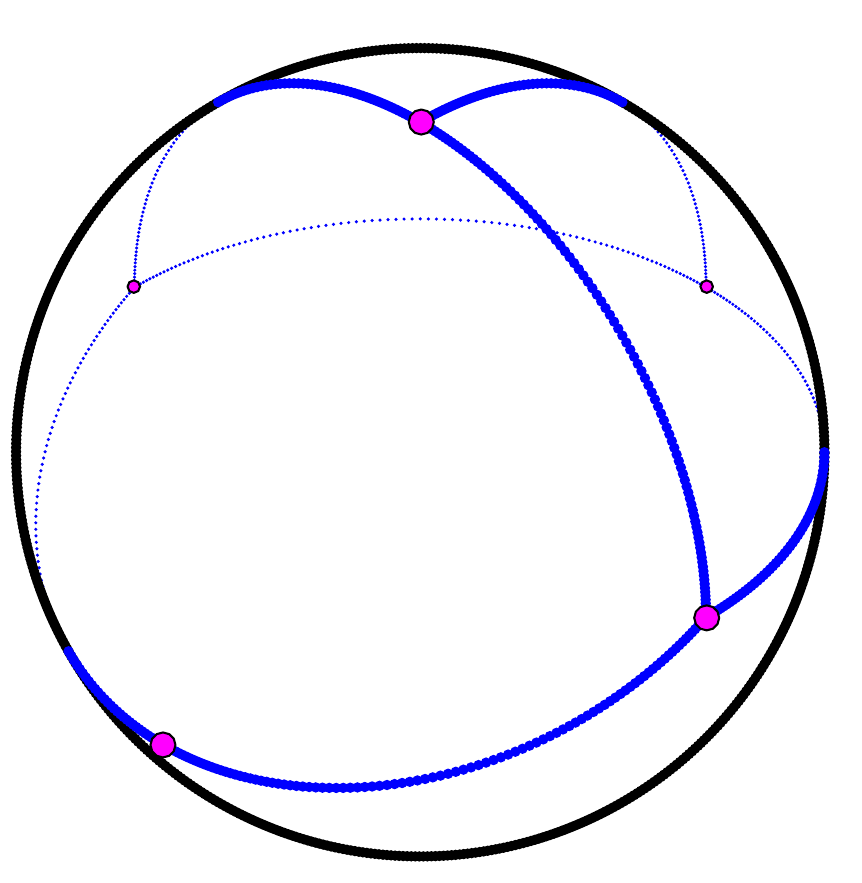}
\includegraphics[width=0.24\textwidth,trim={4cm 7cm 3cm 6cm},clip]{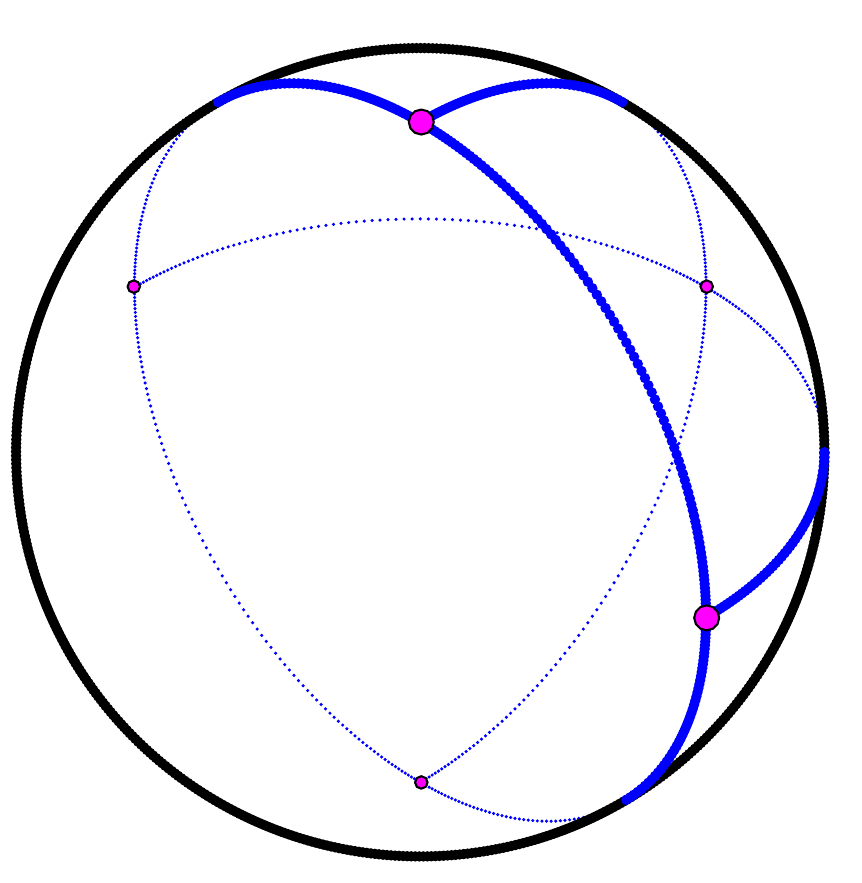}
\end{center}
\caption{\label{fig.orthoplex}
Spherical codes that achieve equality in Rankin's orthoplex bound.
We draw arcs between code points of geodesic distance $\pi/2$, i.e., the minimum distance of the code.
\textbf{(left)}
Six vertices of the octahedron (the three-dimensional orthoplex).
This code is optimal and unique up to isometry.
\textbf{(middle left)}
Removing a vertex from the octahedron produces a size-$5$ spherical code that also achieves equality in Rankin's orthoplex bound.
Such spherical codes enjoy a connected configuration space.
For example, any one of the four points on the equator is able to move away from the fifth point at the north pole.
This motion produces \textbf{(middle right)} and continues until reaching \textbf{(right)}.
Notice that for each of these codes, there exists a pair of antipodes, and the entire code is invariant under the reflection that swaps these antipodes.
The characterization in Theorem~\ref{thm.orthoplex characterization} implies that every optimal spherical code of size $5$ exhibits such a symmetry.}
\end{figure}

The uniquely optimal spherical code of size $2d$ is invariant under a representation of the signed permutation group $B_d$.
Curiously, every optimal code of size $2d-1$ is also invariant under an isometry of the sphere.
Indeed, if $|X|=2d-1$, then $k=d-1$, and since $l\geq k$ and $|X|\geq 2l$, we also have $l=d-1$.
Since $|X_i|\geq2$ for each $i\in[l]$, it then follows that $|X_i|=2$ for at least $l-1$ of these $i$.
Then $X_i$ is a pair of antipodes in $X$, and considering (iii), all other members of $X$ are orthogonal to $\operatorname{span}X_i$.
It follows that $X$ is invariant under the reflection that swaps the members of $X_i$.
(One may also take any signed permutation of the $l-1$ different $X_i$'s of size $2$.)

The above discussion of spherical codes can be distilled into a few key observations:
\begin{itemize}
\item
Every known optimal spherical code is either highly symmetric or has large complexity.
\item
The orthoplex is the uniquely optimal spherical code of size $2d$.
It is highly symmetric and has low complexity.
\item
Every optimal spherical code of size $2d-1$ is invariant under a nontrivial isometry of the sphere.
\end{itemize}
In the next section, we draw inspiration from these observations to predict analogous behaviors for optimal codes in a wide variety of compact metric spaces.

\section{Main definitions and conjectures}

In the previous section, we isolated three key observations about spherical codes.
The least rigorous aspect of these observations is the word ``complexity.''
Intuitively, the complexity of an object corresponds to how difficult it is to express.
In this section, we leverage ideas from computational complexity to make this notion rigorous.
In doing so, we obtain a notion of complexity for \textit{sequences} of objects.
Specifically, a sequence of objects is considered to have ``low'' complexity if there exists an algorithm that constructs those objects in polynomial time.
Taking inspiration from the orthoplex, we consider compact metric spaces that similarly contain uniquely optimal codes of low complexity:

\begin{definition}[informal version]
A \textbf{unicorn space} is a compact metric space $(M,d)$, where $d$ is computable in polynomial time, for which there exists a sequence $n_1<n_2<\cdots$ in $\mathbb{N}$ such that
\begin{itemize}
\item[(i)]
for each $i\in\mathbb{N}$, there exists a unique (up to isometry) optimal code $C_i$ of size $n_i$ in $(M,d)$, and
\item[(ii)]
there exists an algorithm that computes each $C_i$ in polynomial time.
\end{itemize}
We call $\{C_i\}_{i=1}^\infty$ a \textbf{unicorn sequence}.
\end{definition}

Unfortunately, we do not expect the sphere to be a unicorn space, as there does not appear to be a sequence of polynomial time--computable codes that are uniquely optimal.
This departure is an artifact of our rigorous notion of complexity, which requires access to a sequence of objects.
Next, in order for ``polynomial time'' in the above definition to make sense, we need to decide on an encoding of $(M,d)$ that a computer can interact with (e.g., strings over some alphabet).
What follows is a more careful choice of definition.
Here, $\Sigma^*$ denotes the set of all finite strings from an alphabet $\Sigma$.

\begin{definition}[formal version]
\label{def.formal}
Given a finite alphabet $\Sigma$, consider a set $S\subseteq\Sigma^*$ with pseudometric $d$ and a sequence $n_1<n_2<\cdots$ in $\mathbb{N}$ such that
\begin{itemize}
\item[(U1)]
the completion $(M,d)$ of the induced metric space $(S/{\sim},d)$ is compact, where $\sim$ is the relation induced by the pseudometric $d$,
\item[(U2)]
there exists an algorithm that on input $(x,y,k)$ with $x,y\in S$ and $k\in\mathbb{N}$ takes $\operatorname{poly}(\operatorname{length}(x),\operatorname{length}(y),k)$ runtime to return $d_0$ such that $|d_0-d(x,y)|\leq 2^{-k}$,
\item[(U3)]
for each $i\in\mathbb{N}$, there exists a unique (up to isometry) optimal code $C_i$ of size $n_i$ in $(M,d)$, and
\item[(U4)]
there exists an algorithm that on input $(i,k)\in\mathbb{N}^2$ takes $\operatorname{poly}(n_i, k)$ runtime to return a size-$n_i$ subset $C\subseteq S$ such that the induced Hausdorff distance between $C/{\sim}$ and $C_{i}$ is at most $2^{-k}$.
\end{itemize}
Then we call $(M,d)$ a \textbf{unicorn space} with \textbf{unicorn sequence} $\{C_i\}_{i=1}^\infty$.
\end{definition}

\begin{example}
To illustrate the (opaque) formalism of Definition~\ref{def.formal}, let us discuss how to use it to encode the unit sphere in $\mathbb{R}^3$ with chordal distance.
Let $S$ denote the set of strings like
\[
\texttt{(1/23456,-100/1,67/101)}
\]
that describe nonzero triples of rational numbers (in reduced form).
To be extremely explicit, let $\operatorname{dec}\colon S\to\mathbb{Q}^3\setminus\{(0,0,0)\}$ denote the implied ``decoder'', e.g.,
\[
\operatorname{dec}(\texttt{(1/23456,-100/1,67/101)})
=(\tfrac{1}{23456},-100,\tfrac{67}{101}).
\]
(Since we restricted $S$ to only involve representations of rational numbers in reduced form, it holds that $\operatorname{dec}$ is a bijection; this is not terribly important, but it might help to think of $S$ as the set $\mathbb{Q}^3\setminus\{(0,0,0)\}$.)
Now consider the pseudometric $d\colon S\times S\to\mathbb{R}$ defined by
\[
d(x,y)
=\bigg\|\frac{\operatorname{dec}(x)}{\|\operatorname{dec}(x)\|}-\frac{\operatorname{dec}(y)}{\|\operatorname{dec}(y)\|}\bigg\|.
\]
Notably, this choice of $d$ satisfies~(U2).
This is not a metric, but rather a pseudometric since $d(x,y)=0$ whenever $\operatorname{dec}(x)$ and $\operatorname{dec}(y)$ are rational multiples of each other.
Writing $x\sim y$ when $d(x,y)=0$, then $(S/{\sim},d)$ is a metric space that is isometric to the subset $Q\subseteq\mathbb{R}^3$ obtained by normalizing nonzero rational vectors.
Considering $Q$ is dense in the unit sphere, the completion $(M,d)$ of $(S/{\sim},d)$ is isometric to the unit sphere.
This explains~(U1).
Importantly, for every point $u$ in the unit sphere and every $\epsilon>0$, there exists a string in $S$ that represents a point in the unit sphere whose distance from $u$ is less than $\epsilon$.

As such, the only thing preventing the sphere from being a unicorn space is the existence of a unicorn sequence, namely, a sequence $\{C_i\}_{i=1}^\infty$ of codes that together satisfies~(U3) and~(U4).
Here, (U3) requires each $C_i$ to be uniquely optimal.
This is not terribly restrictive of the sphere, since one might expect it to exhibit uniquely optimal codes for infinitely many sizes, though even this is not known.
But then (U4) requires the entire sequence $\{C_i\}_{i=1}^\infty$ to exhibit low complexity, specifically low \textit{computational} complexity using the formalism afforded by building up $(M,d)$ from a set of strings.
This is likely asking too much of the sphere, though proving this rigorously is probably about as hard as Smale's 7th problem~\cite{Smale:98}.
(In any case, it is famously hard to establish whether certain computational problems have a polynomial-time solution.)
\end{example}

In many cases, the uniqueness required in (U3) is spoiled by a feature known as \textit{rattle}.
In particular, given an optimal code $C$ in a compact metric space $(M,d)$, we say $x\in C$ is a \textbf{rattler} if for every $\epsilon>0$, there exists $x'\in M\setminus\{x\}$ with $d(x',x)<\epsilon$ such that $(C\setminus\{x\})\cup\{x'\}$ is also an optimal code.
For example, the last three displays of Figure~\ref{fig.orthoplex} are obtained by ``rattling'' a rattler along a great circle arc.

We are now ready to formulate our first prediction, which is inspired by the fact that the orthoplex is uniquely optimal, has low complexity, and enjoys a large symmetry group: 

\begin{conjx}
\label{conj.1}
Given a unicorn space with nontrivial isometry group $G$, for any unicorn sequence $\{C_i\}_{i=1}^\infty$, there exists $i_0$ such that for every $i>i_0$, there exists $g_i\in G\setminus\{\operatorname{id}\}$ such that $g_iC_i=C_i$.
\end{conjx}

Next, we are interested in generalizing the phenomenon illustrated in Figure~\ref{fig.orthoplex} that comes from removing a point from an orthoplex.
Interestingly, there are many settings in which a code appears to remain optimal after removing a point.
For example, this is conjectured to occur in the triangle~\cite{Oler:61}, the triangular torus~\cite{DickensonGKX:11,ConnellyD:14}, and complex projective spaces~\cite{JasperKM:19}.
However, this is not a general phenomenon, considering the optimal codes in the interval or in the circle.
Instead, taking inspiration from Elekes--Sharir~\cite{ElekesS:10} and Guth--Katz~\cite{GuthK:15}, we pose a general conjecture in terms of \textit{partial symmetries} of the code $C$, that is, nontrivial members $g$ of the isometry group for which $|gC\cap C|$ is large.
First, we identify what it means for this intersection to be large:

\begin{definition}
The \textbf{symmetry strength} of a metric space $(M,d)$ with isometry group $G$ is the supremum of $t$ such that for every size-$t$ subset $T\subseteq M$, there exists $g\in G\setminus\{\operatorname{id}\}$ such that $gT=T$.
Here, we adopt the convention that $g\varnothing=\varnothing$ so that symmetry strength is nonnegative whenever the isometry group is nontrivial.
\end{definition}

For example, consider the Euclidean plane.
In this case, the isometry group is the Euclidean group $\operatorname{E}(2)$.
For every set $T=\{x\}$ of size $1$, one may nontrivially rotate about $x$, and so the symmetry strength is at least $1$.
In fact, for every set $T=\{x,y\}$ of size $2$, one may reflect about the affine line containing $x$ and $y$, and so the symmetry strength is at least $2$.
However, there exist sets $T=\{x,y,z\}$ of size $3$ for which $gT=T$ only if $g=\operatorname{id}$.
Explicitly, one may take $x=(0,0)$, $y=(1,0)$ and $z=(0,2)$.
Consider the triangle whose vertices are given by $T$.
Then since the edge lengths of this triangle are distinct, it follows that $gT=T$ only if $gx=x$, $gy=y$, and $gz=z$.
Since $g(0,0)=gx=x=(0,0)$, it follows that $g\in\operatorname{O}(2)\leq\operatorname{E}(2)$.
Finally, since $y$ and $z$ form a basis, $gy=y$ and $gz=z$ together imply that $g=\operatorname{id}$.
Overall, the symmetry strength of the Euclidean plane is exactly $2$.
Similarly, the symmetry strength of the $k$-dimensional Euclidean space is exactly $k$.

\begin{conjx}
\label{conj.2}
Let $(M,d)$ be a unicorn space with unicorn sequence $\{C_i\}$, nontrivial isometry group $G$, and symmetry strength $t\in[0,\infty)$.
There exists an $i_0$ such that for every $i>i_0$, if $C$ is an optimal code in $(M,d)$ of size $|C_i|-1$, then there exists $g\in G\setminus\{\operatorname{id}\}$ for which $|gC\cap C|>t$.
\end{conjx}

Now that we have formulated explicit predictions (namely, Conjectures~\ref{conj.1} and~\ref{conj.2}), we can test these predictions in various spaces.
We do this in the next section.

\section{Examples and partial results}

\subsection{The interval}
\label{sec.interval}

We start by verifying that the (unit) interval is a unicorn space.
For the sake of clarity, we follow Definition~\ref{def.formal}, but for subsequent examples, we will only verify (U3).
Take $\Sigma:=\{\texttt{0},\texttt{1}\}$ and let $S$ denote the strings of characters from $\Sigma$ whose last character is \texttt{1}.
We will abuse notation by identifying $S$ with the subset of $[0,1)$ with terminating binary expansions.
Define $d(x,y):=|x-y|$ (i.e., our pseudometric is a metric in this case), and put $n_i:=i+1$ for every $i\in\mathbb{N}$.
Following (U1), we have $M=[0,1]$.
For (U2), we may ignore the input $k$ and compute $d_0:=\max(x,y)-\min(x,y)$ in $O(\operatorname{length}(x)+\operatorname{length}(y))$ time.
For (U3), we observe that the unique optimal code of size $n_i=i+1$ is given by $C_i:=\{j/i:j\in\{0,\ldots,i\}\}$.
To see this, put $n:=n_i$ and let $x_1<\cdots< x_n$ denote the members of a code $C\subseteq M$.
Then
\begin{equation}
\label{eq.pigeonhole for segment}
\delta(C)
=\min_{j\in[n-1]}(x_{j+1}-x_j)
\leq\frac{1}{n-1}\sum_{j=1}^{n-1}(x_{j+1}-x_j)
=\frac{x_{n}-x_1}{n-1}
\leq\frac{1}{n-1}, 
\end{equation}
with equality precisely when $C$ takes the claimed form.
Finally, we may compute truncated binary expansions of $j/i$ for each $j$ in polynomial time, and this gives (U4).

The only nontrivial isometry of $[0,1]$ is reflection $g\colon x\mapsto 1-x$.
Since each $C_i$ exhibits this reflection symmetry, it follows that Conjecture~\ref{conj.1} holds in this case.
Next, the symmetry strength in this case is $0$ since, for example, $g(1/3)=2/3\neq1/3$.
Meanwhile, for each $i>1$, the unique optimal code $C_{i-1}$ of size $n_i-1=i\geq2$ has the property that $|gC_{i-1}\cap C_{i-1}|=|C_{i-1}|=i\geq2>0$, and so Conjecture~\ref{conj.2} also holds in this case.
Notice that Conjecture~\ref{conj.2} follows from Conjecture~\ref{conj.1} in this case since $n_i-1=n_{i-1}$.

\subsection{The circle}
\label{sec.circle}

Consider the circle with geodesic distance.
Without loss of generality, we put
\[
M:=\{(\cos\theta,\sin\theta):\theta\in[0,2\pi)\}.
\]
The isometry group of $M$ is the orthogonal group $\operatorname{O}(2)$.
A pigeonhole argument similar to \eqref{eq.pigeonhole for segment} proves that $n$ uniformly spaced points form an optimal code, and this optimizer is unique up to isometry.
It follows that the circle is a unicorn space.
Every optimal code is invariant under reflection about the line connecting any one of its points to the origin, and so Conjecture~\ref{conj.1} holds.
Next, the symmetry strength of the circle is at least $2$ since for any $x,y\in M$, the set $\{x,y\}$ is invariant under reflection about the line that bisects the angle between $x$ and $y$.
Meanwhile, given a generic size-$3$ subset $T=\{x,y,z\}\subseteq M$, since $d(x,y)\neq d(y,z)\neq d(z,x)\neq d(x,y)$, it holds that $gT=T$ with $g\in\operatorname{O}(2)$ only if $gx=x$, $gy=y$ and $gz=z$.
By genericity, $\{x,y\}$ forms a basis for $\mathbb{R}^2$, and so $g=\operatorname{id}$.
As such, the symmetry strength is exactly $2$.
For every optimal code $C$ of size at least $3$, we may take $g$ to be reflection about the line spanned by one of its points, and then $|gC\cap C|=|C|\geq3>2$.
As such, Conjecture~\ref{conj.2} also holds in this case.
(Like the interval case, since $n_i-1=n_{i-1}$, Conjecture~\ref{conj.2} is less interesting.)

\subsection{The equilateral triangle}

Consider the region $M$ bounded by an equilateral triangle with Euclidean distance.
The isometry group of this space is the dihedral group $D_3$.
Melissen~\cite{Melissen:93} established that whenever $n$ is a triangular number $T(k):=\binom{k+1}{2}$, the unique optimal code of size $n$ comes from the hexagonal lattice.
These codes are invariant under the isometry group, and so Conjecture~\ref{conj.1} holds.
The symmetry strength of $M$ is $0$ since the $D_3$-orbit of a generic point in $M$ has size $|D_3|$.
Erd\H{o}s and Oler~\cite{Oler:61} conjecture that the optimal codes of size $T(k)-1$ are obtained by removing any point from the optimal code of size $T(k)$.
Conditional on this conjecture, then for every optimal code $C$ of size $T(k)-1$, every $g\in D_3\setminus\{\operatorname{id}\}$ has the property that $|gC\cap C|\geq|C|-2=T(k)-3$, which is strictly positive for every $k\geq3$.
As such, Conjecture~\ref{conj.2} (conditionally) holds in this case.
In addition, Lubachevsky, Graham and Stillinger \cite{LubachevskyGS:97} conjecture the form of all optimal codes of size
\[
T((k+1)p-1)+(2p+1)T(k).
\]
Notice that for $p=0$, this size is $T(k)$, for which the optimal codes are already known.
For each $p\in\mathbb{N}$, the conjectured optimal codes include $p-1$ rattlers, and so uniqueness only holds for $p=1$.
In the case $p=1$, the conjectured optimal codes exhibit full $D_3$ symmetry, which agrees with Conjecture~\ref{conj.1}.

\subsection{Orthotopes}
\label{sec.orthotopes}

Let $\leq$ denote the entrywise partial order in $\mathbb{R}^m$, select $u\in\mathbb{R}^m$ with $u\geq0$, and consider the orthotope $M:=\{x\in \mathbb{R}^m:0\leq x\leq u\}$ with distance induced by the $\infty$-norm.
The isometry group of this space is generated by all coordinate permutations that fix $u$ and all coordinate reflections of the form $x_i\mapsto u_i-x_i$. 
In the following result, we characterize when $M$ is a unicorn space.
First, we say $x\in\mathbb{R}$ is a \textbf{divisor} of $y\in\mathbb{R}$ if $y/x\in\mathbb{Z}$.
A collection $u$ of real numbers is said to be \textbf{commensurable} if they have a common divisor, in which case we let $\operatorname{gcd}(u)$ denote the greatest common divisor.

\begin{theorem}
\label{thm.unicorn orthotopes}
Select $u\in\mathbb{R}^m$ with $u\geq0$ and put $M:=\{x\in \mathbb{R}^m:0\leq x\leq u\}$.
Then $(M,\|\cdot\|_\infty)$ is a unicorn space if and only if the coordinates of $u$ are commensurable.
In this case, the following are equivalent:
\begin{itemize}
\item[(i)]
The optimal code of size $n$ in $(M,\|\cdot\|_\infty)$ is unique up to isometry.
\item[(ii)]
The optimal code of size $n$ in $(M,\|\cdot\|_\infty)$ is unique.
\item[(iii)]
$n=\prod_{i=1}^m(\frac{u_i}{\delta}+1)$ for some divisor $\delta$ of $\operatorname{gcd}(u)$.
\end{itemize}
\end{theorem}

Our proof follows from a couple of lemmas:

\begin{lemma}
\label{lem.unicorns in a jagged orthotope}
Select $u^{(1)},\ldots,u^{(k)}\in(\mathbb{N}\cup\{0\})^m$ and put
\begin{equation}
\label{eq.precube}
M:=\bigcup_{j=1}^k\{x\in \mathbb{R}^m:0\leq x\leq u^{(j)}\}.
\end{equation}
Then $M\cap\mathbb{Z}^m$ is the unique code in $(M,\|\cdot\|_\infty)$ of its size with minimum distance $1$.
\end{lemma}

\begin{proof}
As an intermediate result, we establish that a code $C$ of size $|M\cap\mathbb{Z}^m|$ has minimum distance $1$ only if
\begin{equation}
\label{eq.disjoint union}
\operatorname{cl}\bigg(\bigsqcup_{x\in C}B(x,\tfrac{1}{2})\bigg)
=\operatorname{cl}\Big(M+B(0,\tfrac{1}{2})\Big).
\end{equation}
To see this, first note that the code $C=M\cap\mathbb{Z}^m$ satisfies \eqref{eq.disjoint union}, and so we have $|M\cap\mathbb{Z}^m|=\operatorname{vol}(M+B(0,\tfrac{1}{2}))$.
If $C$ has size $|M\cap\mathbb{Z}^m|$ and minimum distance $1$, then
\[
\operatorname{vol}(M+B(0,\tfrac{1}{2}))
=|M\cap\mathbb{Z}^m|
=|C|
=\sum_{x\in C}\operatorname{vol}(B(x,\tfrac{1}{2}))
\leq\operatorname{vol}(M+B(0,\tfrac{1}{2})).
\]
Then \eqref{eq.disjoint union} follows from the fact that the above inequality is saturated.

To prove that $M\cap\mathbb{Z}^m$ is the only code of its size in $M$ with minimum distance $1$, we prove something stronger.
In particular, we prove that it is the only code of its size with minimum distance \textit{at least} $1$.
We induct on $|M\cap\mathbb{Z}^m|$.
First, if $|M\cap\mathbb{Z}^m|=2$, then $M$ is isometric to $[0,1]$, and indeed, the unique code of size $2$ with (minimum) distance at least $1$ is $\{0,1\}$.
Now fix $r\in\mathbb{N}$ with $r\geq 2$, suppose our claim holds for all $M'$ such that $|M'\cap\mathbb{Z}^m|=r$, and consider any $M$ with $|M\cap\mathbb{Z}^m|=r+1$.
Let $C$ denote any code in $M$ with minimum distance at least $1$.
Select any $j\in[k]$ for which $u^{(j)}$ is maximal in $M$ under the entrywise partial order and put $w=u^{(j)}+\frac{1}{2}\cdot \mathbf{1}\in\operatorname{cl}(M+B(0,\tfrac{1}{2}))$, where $\mathbf{1}$ denotes the all-ones vector.
Then $u^{(j)}$ is the unique point $u\in M$ such that $w\in\operatorname{cl}(B(u,\tfrac{1}{2}))$, and so \eqref{eq.disjoint union} implies $u^{(j)}\in C$.
Put $M':=M\cap\{x\in\mathbb{R}^m: x\not>u^{(j)}-\mathbf{1}\}$, which can be expressed in the form \eqref{eq.precube}.
Then $C\setminus\{u^{(j)}\}\subseteq M'$ has minimum distance at least $1$ and the maximality of $u^{(j)}$ implies
\[
|C\setminus\{u^{(j)}\}|
=|M\cap\mathbb{Z}^m|-1
=|M'\cap\mathbb{Z}^m|.
\]
The induction hypothesis then gives
\[
C
=\{u^{(j)}\}\cup(C\setminus\{u^{(j)}\})
=\{u^{(j)}\}\cup(M'\cap\mathbb{Z}^m)
=M\cap\mathbb{Z}^m.
\qedhere
\]
\end{proof}

\begin{lemma}
\label{lem.unicorns in a orthotope}
Select $u\in(\mathbb{N}\cup\{0\})^m$ and put $M:=\{x\in \mathbb{R}^m:0\leq x\leq u\}$.
Then $M\cap\mathbb{Z}^m$ is the unique optimal code in $(M,\|\cdot\|_\infty)$ of size $\prod_{i=1}^m(u_i+1)$.
\end{lemma}

\begin{proof}
We will show that $M\cap\mathbb{Z}^m$ is an optimal code of its size, and then uniqueness follows from Lemma~\ref{lem.unicorns in a jagged orthotope}.
First, the minimum distance of $M\cap\mathbb{Z}^m$ is $1$.
Now let $C\subseteq M$ be any code with minimum distance $\delta$.
Then the open $\infty$-balls $B(x,\frac{\delta}{2})$ of radius $\frac{\delta}{2}$ centered at each $x\in C$ are disjoint and reside in the Minkowski sum $M+B(0,\frac{\delta}{2})$.
We compare volumes to obtain the bound
\[
|C|
\leq\frac{\operatorname{vol}(M+B(0,\frac{\delta}{2}))}{\operatorname{vol}(B(0,\frac{\delta}{2}))}
=\frac{\prod_{i=1}^m(u_i+\delta)}{\delta^m}
=\prod_{i=1}^m\bigg(\frac{u_i}{\delta}+1\bigg).
\]
In particular, $\delta>1$ implies $|C|<\prod_{i=1}^m(u_i+1)$.
The contrapositive reveals that $M\cap\mathbb{Z}^m$ is an optimal code of its size.
\end{proof}

\begin{proof}[Proof of Theorem~\ref{thm.unicorn orthotopes}]
If the coordinates of $u$ are commensurable, then we can scale $u$ in such a way that Lemma~\ref{lem.unicorns in a orthotope} gives (iii)$\Rightarrow$(ii), while (ii)$\Rightarrow$(i) is immediate.
This combined with the explicit construction in Lemma~\ref{lem.unicorns in a orthotope}, implies that $(M,\|\cdot\|_\infty)$ is a unicorn space.
Now suppose the optimal code of size $n$ in $(M,\|\cdot\|_\infty)$ is unique up to isometry (and $u$ is not necessarily commensurable).
Since the isometry group is discrete, this implies that the optimal code does not exhibit rattle.
Put $\delta=\delta(C)$ and select any $i\in[m]$.
Then there exists $x^{(0)}\in C$ with $x^{(0)}_i=0$, since otherwise any $x\in C$ with the smallest $i$-coordinate can rattle along the segment toward the point $x'$ that equals $x$ in every coordinate except $x'_i=0$.
Next, there exists $x^{(1)}$ such that $\|x^{(1)}-x^{(0)}\|_\infty=\delta$ and $x^{(1)}_i=\delta$, since otherwise $x^{(0)}$ can rattle along the segment toward the point $x'$ that equals $x^{(0)}$ in every coordinate except $x'_i=\delta$.
In this way, we iteratively obtain $x^{(j+1)}$ such that $\|x^{(j+1)}-x^{(j)}\|_\infty=\delta$ and $x^{(j+1)}_i=(j+1)\delta$.
This process ends with $j_0$ such that $x^{(j_0)}_i>u_i-\delta$.
It must hold that $x^{(j_0)}_i=u_i$, since otherwise $x^{(j_0)}$ can rattle along the segment toward the point $x'$ that equals $x^{(j_0)}$ in every coordinate except $x'_i=u_i$.
Then $u_i=x^{(j_0)}_i=j_0\delta$, i.e., $\delta$ is a divisor of $u_i$.
Since our choice for $i\in[m]$ was arbitrary, we conclude that the coordinates of $u$ are commensurable, and furthermore, (i)$\Rightarrow$(iii), as desired.
\end{proof}

Up to scaling, every optimal code in $(M,\|\cdot\|_\infty)$ that is unique up to isometry takes the form described in Lemma~\ref{lem.unicorns in a orthotope}.
Furthermore, every code of this form is invariant under the reflection $x_1\mapsto u_1-x_1$ (for example).
It follows that Conjecture~\ref{conj.1} holds for all orthotopes $(M,\|\cdot\|_\infty)$.
Next, we seek to test Conjecture~\ref{conj.2}.
To this end, we start with a lemma:

\begin{lemma}
\label{lem.big codes have known min distance}
Select $u\in\mathbb{N}^m$ and put $M:=\{x\in \mathbb{R}^m:0\leq x\leq u\}$.
Any code $C$ in $(M,\|\cdot\|_\infty)$ with minimum distance $\delta(C)>1$ necessarily satisfies $|C|\leq\prod_{i=1}^m u_i$.
\end{lemma}

\begin{proof}
We induct on $m$.
For $m=1$, the argument in \eqref{eq.pigeonhole for segment} implies that $1<\delta(C)\leq\frac{u_1}{n-1}$, and rearranging gives $|C|=n<u_1+1$, i.e., $|C|\leq u_1$.
Now suppose the claim holds for some $m\geq1$ and select any $u\in\mathbb{N}^{m+1}$.
For each $j\in[u_{m+1}]$, consider the $j$th ``thick slice'' $R_j:=\{x\in M:x_{m+1}\in[j-1,j]\}$, and let $\pi$ denote projection onto the first $m$ coordinates.
Given $C\subseteq M$ with $\delta(C)>1$, then for every $j\in[u_{m+1}]$ and $x,y\in C\cap R_j$ with $x\neq y$, we have $\|x-y\|_\infty=\|\pi(x)-\pi(y)\|_\infty$.
That is, each code $\pi(C\cap R_j)$ has minimum distance $>1$.
By the induction hypothesis, it follows that
\[
|C|
\leq \sum_{j=1}^{u_{m+1}}|C\cap R_j|
=\sum_{j=1}^{u_{m+1}}|\pi(C\cap R_j)|
\leq \sum_{j=1}^{u_{m+1}}\prod_{i=1}^{m}u_i
=\prod_{i=1}^{m+1} u_i.
\qedhere
\]
\end{proof}

\begin{figure}
\begin{center}
\includegraphics[width=0.32\textwidth,trim={4cm 8cm 3cm 7cm},clip]{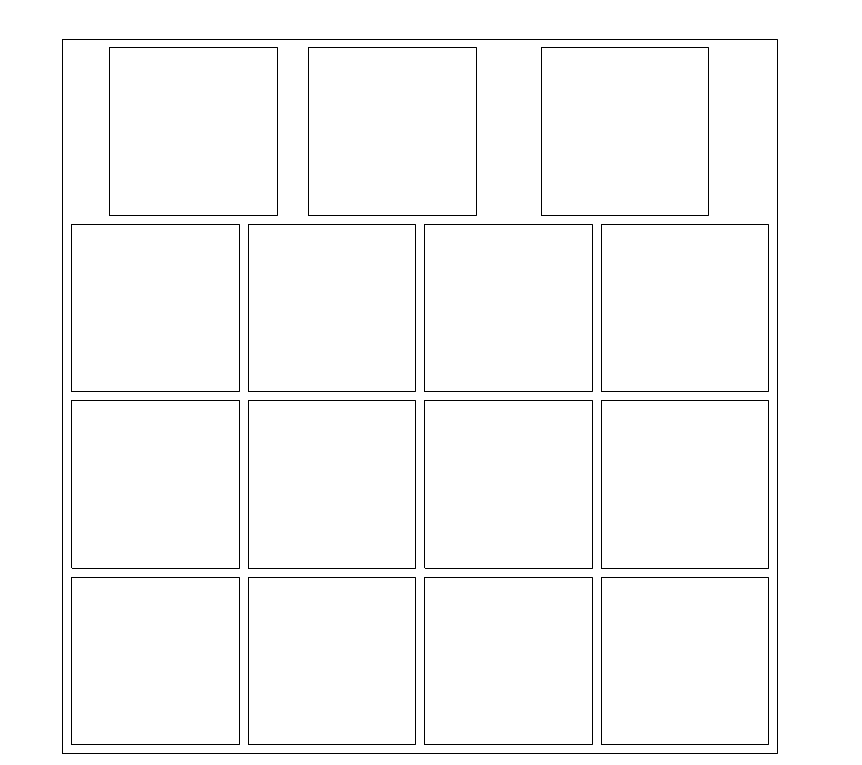}
\includegraphics[width=0.32\textwidth,trim={4cm 8cm 3cm 7cm},clip]{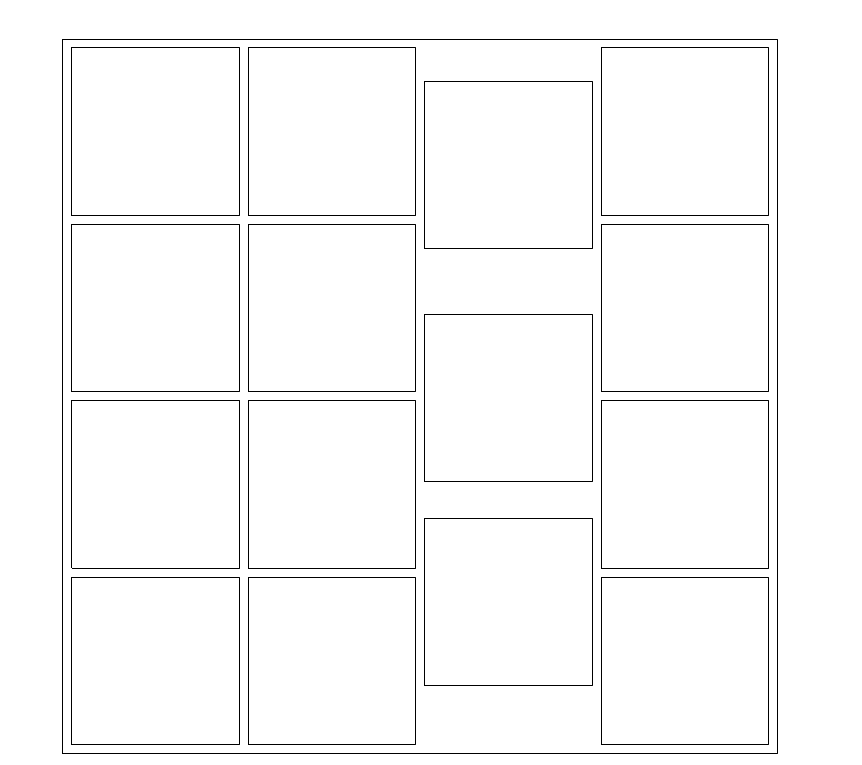}
\includegraphics[width=0.32\textwidth,trim={4cm 8cm 3cm 7cm},clip]{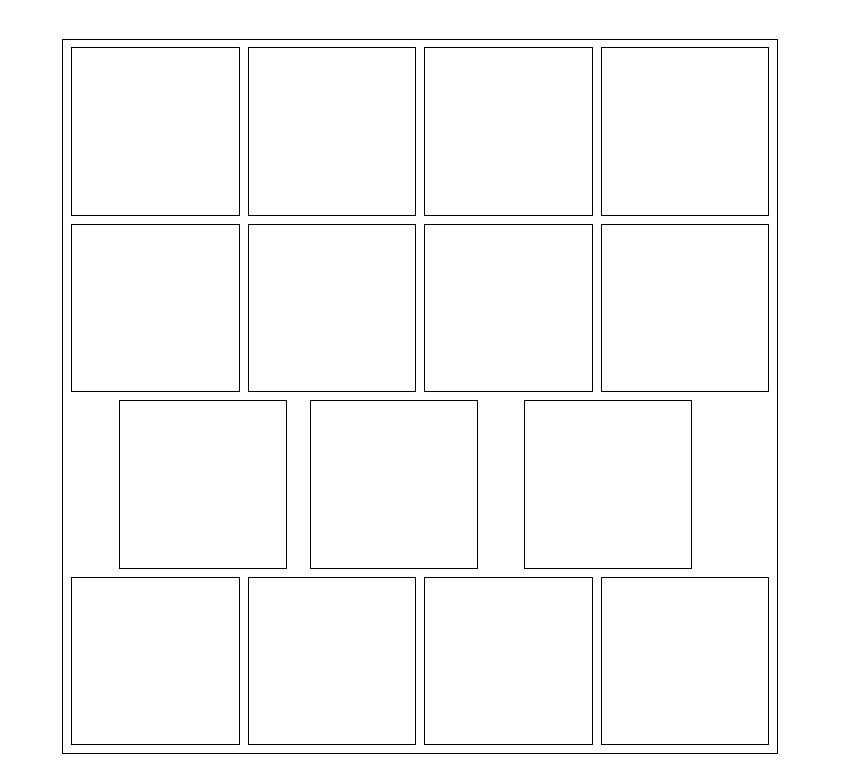}
\end{center}
\caption{\label{fig.puzzle15}
Generic points in the configuration space of Sam Loyd's 15 puzzle.
In each case, twelve of the squares are aligned with a lattice, while the other three are free to slide in their common row or column.
The centers of these squares form an optimal code under the $\infty$-norm in the $3\times3$ square that contains them.
This behavior is a special case of Theorem~\ref{thm.orthotope code minus 1}.
}
\end{figure}

\begin{theorem}
\label{thm.orthotope code minus 1}
Select $u\in(\mathbb{N}\cup\{0\})^m$ and put $M:=\{x\in \mathbb{R}^m:0\leq x\leq u\}$.
Every optimal code $C$ in $(M,\|\cdot\|_\infty)$ of size $\prod_{i=1}^m(u_i+1)-1$ enjoys a decomposition $C=A\sqcup B$ for which there exists $i\in[m]$ such that
\[
|A|=\prod_{s=1}^m(u_s+1)-(u_i+1),
\qquad
|B|=u_i,
\]
and furthermore, $A\subseteq M\cap\mathbb{Z}^m$ and $B\subseteq\{x+te_i:t\in\mathbb{R}\}$ for some $x\in\mathbb{Z}^m$.
\end{theorem}

\begin{proof}
We induct on $\sum_{i=1}^m u_i$.
For $\sum_{i=1}^m u_i=1$, we have $M=[0,1]$ up to isometry, and the codes of size $\prod_{i=1}^m(u_i+1)-1=1$ take the claimed form with $A=\varnothing$.
Now fix $r\in\mathbb{N}$, suppose our claim holds for all $M'$ such that $\sum_{i=1}^m u_i=r$, and consider any $M$ with $\sum_{i=1}^m u_i=r+1$.
By Lemmas~\ref{lem.unicorns in a orthotope} and~\ref{lem.big codes have known min distance} and the monotonicity of the maximum minimum distance, we have $\delta(C)=1$.
Select $i\in[m]$ such that $u_i\geq1$, and for each $j\in\{0,\ldots,u_i\}$, consider the ``thin slice'' $S_{j}:=\{x\in M:x_i=j\}$.

\textbf{Case I:}
There exists $j\in[u_i]$ such that $|C\cap S_{j}|\geq\prod_{s\neq i}(u_s+1)$.
Then Lemma~\ref{lem.unicorns in a orthotope} implies that $C\cap S_{j}=S_{j}\cap\mathbb{Z}^d$.
This in turn implies
\[
\operatorname{cl}\bigg(\bigsqcup_{x\in C\cap S_{j}}B(x,\tfrac{1}{2})\bigg)
=\operatorname{cl}\Big(S_{j}+B(0,\tfrac{1}{2})\Big),
\]
meaning every $x\in C\setminus S_{j}$ resides in $M\setminus(S_{j}+B(0,1))$.
As such, we may identify 
\[
\{x\in M:x_i\in[0,j-\tfrac{1}{2})\cup[j+\tfrac{1}{2},u_i]\}
\]
with the smaller space $M':=\{x\in \mathbb{R}^m:0\leq x\leq u'\}$, where $u'$ equals $u$ in every coordinate except $u'_i=u_i-1$.
The result then follows from our induction hypothesis, since every optimal code in $M'$ takes the form $A\sqcup B$, and so the corresponding optimal codes in $M$ are obtained by contributing $S_j\cap\mathbb{Z}^d$ to $A$.

\textbf{Case II:}
For every $j\in[u_i]$, it holds that $|C\cap S_{j}|<\prod_{s\neq i}(u_s+1)$.
Consider the ``thick slices'' defined by
\[
R_{j}(\ell):=\left\{\begin{array}{cl}
\{x\in M:x_i\in[\ell-1,\ell)\} &\text{if }\ell\leq j\\
\{x\in M:x_i\in(\ell-1,\ell]\} &\text{if }\ell> j.
\end{array}\right.
\]
For each $j\in\{0,\ldots,u_i\}$, this gives a partition
\[
M
=S_j\sqcup R_{j}(1)\sqcup\cdots\sqcup R_{j}(u_i).
\]
Let $\pi$ denote projection onto the coordinates indexed by $[m]\setminus\{i\}$.
Since $\delta(C)=1$ and the slice $R_{j}(\ell)$ is half-open, it holds that every $x,y\in C\cap R_{j}(\ell)$ with $x\neq y$ necessarily satisfies $\|x-y\|_\infty=\|\pi(x)-\pi(y)\|_\infty$.
That is, $\pi(C\cap R_{j}(\ell))$ has minimum distance $\geq1$, and so $|C\cap R_{j}(\ell)|\leq \prod_{s\neq i}(u_s+1)$ as a consequence of Lemma~\ref{lem.unicorns in a orthotope}.
This bound holds for all $\ell$, whereas $|C\cap S_{j}|\leq\prod_{s\neq i}(u_s+1)-1$.
We apply pigeonhole over the $j$th partition of $M$ to conclude that $|C\cap R_{j}(\ell)|= \prod_{s\neq i}(u_s+1)$ for every $\ell$ and $|C\cap S_{j}|=\prod_{s\neq i}(u_s+1)-1$.

Next, take $R_\ell:=\{x\in M:x_i\in(\ell-1,\ell)\}$ and observe that for every $j$ and $\ell$, there exists $j'$ such that $R_j(\ell)=R_\ell\sqcup S_{j'}$.
It follows that $|C\cap R_{\ell}|=|C\cap R_{j}(\ell)|-|C\cap S_{j'}|=1$ for each $\ell\in[u_i]$.
Let $x^{(\ell)}$ denote the unique member of $C\cap R_\ell$.
We will take
\[
B:=\{x^{(\ell)}:\ell\in[u_i]\},
\qquad
A:=C\setminus B.
\]
For every $y\in S_{\ell-1}\cup S_\ell$, it holds that $\|x^{(\ell)}-y\|_\infty=\|\pi(x^{(\ell)})-\pi(y)\|_\infty$, and so $\pi((C\cap S_{\ell-1})\cup\{x^{(\ell)}\})$ and $\pi((C\cap S_{\ell})\cup\{x^{(\ell)}\})$ both have minimum distance $\geq1$.
Lemma~\ref{lem.unicorns in a orthotope} then implies that both codes equal $\pi(M)\cap\mathbb{Z}^{m-1}$.
Comparing neighboring thick slices then gives
\begin{align*}
\pi(x^{(\ell)})
&=\pi((C\cap S_{\ell})\cup\{x^{(\ell)}\})\setminus\pi(C\cap S_\ell)\\
&=\pi((C\cap S_{\ell})\cup\{x^{(\ell+1)}\})\setminus\pi(C\cap S_\ell)
=\pi(x^{(\ell+1)})
\end{align*}
for every $\ell\in[u_i-1]$.
That is, $\pi(x^{(\ell)})=\pi(x^{(\ell')})$ for every $\ell,\ell'\in[u_i]$.
The claim follows.
\end{proof}

Theorem~\ref{thm.orthotope code minus 1} implies that Conjecture~\ref{conj.2} also holds for orthotopes $(M,\|\cdot\|_\infty)$.
Indeed, if $M$ is $1$-dimensional, then $M$ is isometric to an interval, which was treated in Subsection~\ref{sec.interval}.
Otherwise, Theorem~\ref{thm.orthotope code minus 1} describes how every optimal code $C$ of size $n-1$ has a decomposition $C=A\sqcup B$ where $B$ has size $u_i$ and resides in an affine line parallel to the span of some $e_i$.
For this choice of $i\in[m]$, $A$ is invariant to the coordinate reflection $x_i\mapsto u_i-x_i$.
Furthermore, $A$ has size $\prod_{s=1}^m(u_s+1)-(u_i+1)$, which is positive since $u$ has at least $2$ positive entries.
For comparison, the symmetry strength of $M$ is $0$ since a generic member of $M$ is not fixed by any member of $M$'s discrete isomorphism group.

\subsection{Tori}

\subsubsection{Candidate unicorn spaces}

Given a lattice $L$ in $\mathbb{R}^m$, consider the space $M:=\mathbb{R}^m/L$ with metric
\[
d(x,y):=\min\{\|s-t\|_2:s\in x,t\in y\}.
\]
The isometry group of $M$ is $\operatorname{Aut}(L)\ltimes(\mathbb{R}^m/L)$, where $\operatorname{Aut}(L)$ denotes the subgroup of $\operatorname{O}(m)$ that preserves $L$.
We are particularly interested in lattices $L$ that give the unique (up to isometry) densest periodic packing of spheres of some fixed radius in $\mathbb{R}^m$.
This occurs in dimensions $m\in\{1,2,8,24\}$ (see~\cite{Toth:40,Viazovska:17,CohnKMRV:17}) and is conjectured to occur for $m=4$ as well~\cite{Musin:18}.

Following~\cite{ConwayS:13}, we define $A_n\subseteq\operatorname{span}\{1_{n+1}\}^\perp\cong\mathbb{R}^n$ to be the points in $\mathbb{Z}^{n+1}$ whose coordinates sum to zero, define $D_n\subseteq\mathbb{R}^n$ to be the points in $\mathbb{Z}^n$ whose coordinates sum to an even number, define $E_8\subseteq\mathbb{R}^8$ to be points in $\mathbb{Z}^8\cup(\mathbb{Z}+\frac{1}{2})^8$ whose coordinates sum to an even number, and define $\Lambda_{24}\subseteq\mathbb{R}^{24}$ to be the lattice generated by all vectors of the form $\frac{1}{\sqrt{8}}(\mp 3,(\pm1)_{23})$, where the $\mp3$ may be in any position, and the upper signs are taken on the support of a Golay codeword.
Then $A_1$, $A_2$, $E_8$ and $\Lambda_{24}$ give the only known optimal sphere packings in Euclidean spaces, while $D_4$ gives the putatively optimal sphere packing for $\mathbb{R}^4$.
In all such cases, we can use (putatively) optimal sphere packings to produce (putatively) optimal codes in $\mathbb{R}^m/L$.

Throughout this section, we use the following notation:
Given a lattice $L\subseteq\mathbb{R}^m$, we let $\ell(L)$ denote the smallest $\|v\|_2$ over all nonzero $v\in L$, and we let $\overline{L}$ denote the normalized lattice $\ell(L)^{-1}\cdot L$.

\subsubsection{Lattice codes and holy codes}

Select any $T$ in the conformal orthogonal group
\[
\operatorname{CO}(m)
:=\{cQ:c>0,Q\in\operatorname{O}(m)\}
\]
such that $TL$ contains $L$ as a sublattice (e.g., $T=\tfrac{1}{k}I$ for any $k\in\mathbb{N}$).
For any such $T$, we refer to $(TL)/L$ as a \textbf{lattice code} in $\mathbb{R}^m/L$.

\begin{lemma}
\label{eq.lattice codes opt}
Suppose $L\subseteq\mathbb{R}^m$ is a lattice that gives the unique (up to isometry) densest periodic packing of spheres of some fixed radius in $\mathbb{R}^m$.
Then
\begin{itemize}
\item[(a)]
every lattice code in $\mathbb{R}^m/L$ is optimal, and
\item[(b)]
if there exists a lattice code of size $n$ in $\mathbb{R}^m/L$, then every optimal code of size $n$ in $\mathbb{R}^m/L$ is a lattice code up to isometry.
\end{itemize}
\end{lemma}

\begin{proof}
Suppose there is a lattice code of size $n$, and consider any code $C=\{s+L:s\in S\}$ of size $n$, where $S$ is a set of coset representatives.
Then $P(C):=\frac{\ell(L)}{\delta(C)}\bigcup_{s\in S}(s+L)$ gives a periodic packing of spheres of radius $\ell(L)/2$ with density matching that of the periodic packing $\bigcup_{s\in S}(s+L)$ of spheres of radius $\delta(C)/2$:
\[
\frac{n\cdot V_m\cdot (\delta(C)/2)^m}{\operatorname{covolume}(L)},
\]
where $V_m$ denotes the volume of a ball in $\mathbb{R}^m$ of radius $1$, and $\operatorname{covolume}(L)$ refers to the volume of a fundamental domain of $L$.
Observe that the density of $P(C)$ increases as $\delta(C)$ increases.
Taking $T=cQ\in\operatorname{CO}(m)$ such that $TL\geq L$ and $|TL:L|=n$ then gives a lattice code $C=(TL)/L$ of size $n$ such that $\delta(C)=c\ell(L)$, in which case $P(C)=QL$.
The optimality of $QL$ as a sphere packing then implies the optimality of $C$ as a code.
This gives (a), whereas (b) follows from the hypothesis that $L$ is unique up to isometry as the densest periodic packing of spheres.
\end{proof}

We seek a unicorn sequence of lattice codes, and so we need to determine when lattice codes are unique up to isometry.
Perhaps surprisingly, it is possible to have inequivalent lattice codes of the same size; see Figure~\ref{fig.lattice} for an example where $L=A_2$ and $n=49$.
In pursuit of uniqueness, we start with a nonexistence result:

\begin{figure}
\begin{center}
\includegraphics[width=0.45\textwidth,trim={4cm 10cm 3cm 9cm},clip]{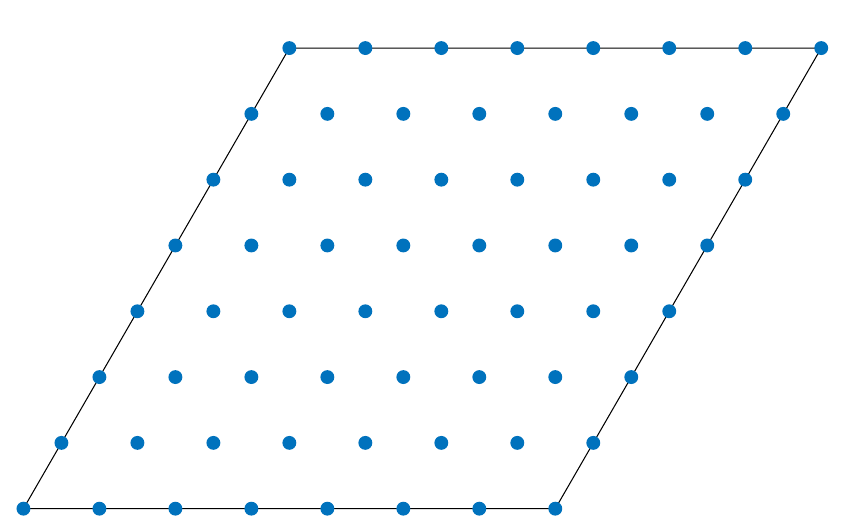}
\includegraphics[width=0.45\textwidth,trim={4cm 10cm 3cm 9cm},clip]{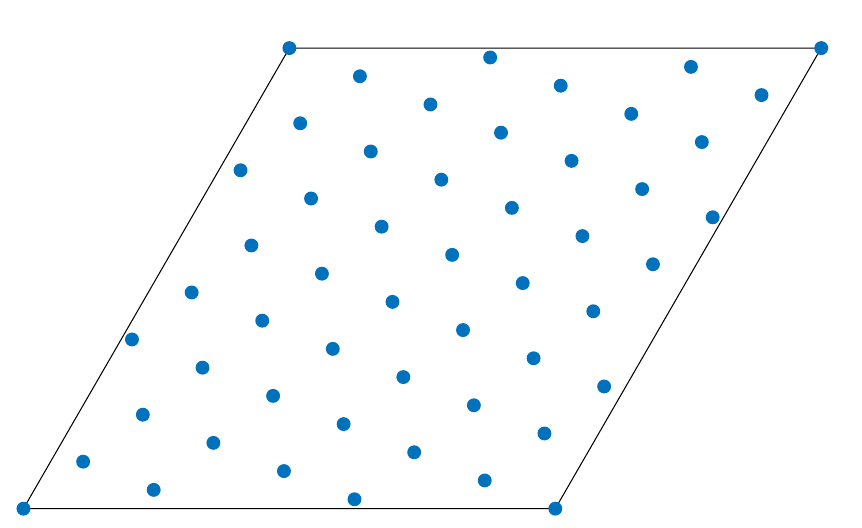}
\end{center}
\caption{\label{fig.lattice}
Inequivalent optimal codes of size $49$ in $\mathbb{R}^2/A_2$; i.e., there is no isometry of $\mathbb{R}^2/A_2$ that maps one code to the other.
Note that the four vertices on the corners of the fundamental domain are identified as the same point; similarly, in the left-hand plot, we identify the appropriate pairs of points on parallel edges.
}
\end{figure}

\begin{lemma}
For any lattice $L\subseteq\mathbb{R}^m$, if $C$ is a lattice code in $\mathbb{R}^m/L$, then
\[
|C|\in N(L):=\{\|z\|_2^m:z\in\overline{L}\setminus\{0\}\}.
\]
\end{lemma}

\begin{proof}
Given a lattice code $C$, there exists $T=cQ\in\operatorname{CO}(m)$ such that $TL\geq L$ and $C=(TL)/L$.
Select any nonzero $v\in L$ of norm $\ell:=\ell(L)$ and put $u:=T^{-1}v$.
As such, $\ell=\|v\|_2=\|Tu\|_2=\|cQu\|_2=c\|u\|_2$, and so
\[
|C|
=|TL:L|
=\frac{\operatorname{covolume}(L)}{\operatorname{covolume}(TL)}
=|\operatorname{det}(T^{-1})|
=c^{-m}
=(\ell^{-1}\|u\|_2)^m
\in N(L).
\qedhere
\]
\end{proof}

Taking inspiration from~\cite{ConnellyD:14}, lattice codes of these sizes are easy to find in many cases:

\begin{lemma}
\label{lem.dilated lattices}
Select $L\in\{A_1,A_2,D_4,E_8\}$, put $m:=\operatorname{dim}L$.
For every $n\in N(L)$, there exists a lattice code in $\mathbb{R}^m/L$ of size $n$.
\end{lemma}

\begin{proof}
Each of these lattices exhibits hidden multiplicative structure.
Specifically, $\overline{A}_1$ is isometric to the integers in $\mathbb{R}$, $\overline{A}_2$ is isometric to the Eisenstein integers in $\mathbb{C}$, $\overline{D}_4$ is isometric to the Hurwitz integral quaternions in $\mathbb{H}$, and $\overline{E}_8$ is isometric to the Coxeter--Dickson integral octonions in $\mathbb{O}$; see~\cite{ConwayS:03} for details.
This identification of $\mathbb{R}^m$ with a Euclidean Hurwitz algebra with subring $\overline{L}$ endows $\mathbb{R}^m$ with a multiplication $(x,y)\mapsto xy$ such that
\begin{itemize}
\item[(i)]
$\|xy\|_2=\|x\|_2\|y\|_2$ for every $x,y\in\mathbb{R}^m$, and
\item[(ii)]
$xy\in \overline{L}$ for every $x,y\in \overline{L}$.
\end{itemize}
Select any nonzero $z\in \overline{L}$.
Then (i) implies that the map $[z]\colon x\mapsto zx$ takes the form $[z]=\|z\|_2Q$ for some $Q\in\operatorname{O}(m)$, and by (ii), we further have that the image $L'$ of $L$ under $[z]$ is a subset of $L$.
Put $T=[z]^{-1}$.
Then $L=TL'$ is a sublattice of $TL$, and furthermore,
\[
|TL:L|
=\frac{\operatorname{covolume}(L)}{\operatorname{covolume}(TL)}
=|\operatorname{det}[z]|
=\|z\|_2^m.
\qedhere
\]
\end{proof}

Notice that $N(A_1)=\mathbb{N}$, which matches our analysis of the circle.
Next, the set $N(A_2)=\{1, 3, 4, 7, 9, 12, 13,\ldots\}$ is known as the Loeschian numbers (OEIS A003136).
(This sequence of code sizes was derived using a related approach in~\cite{ConnellyD:14}.)
Observe that for every nonzero $z\in \overline{D}_4$, it holds that $\|z\|_2^2\in\mathbb{N}$.
Also, by virtue of its isometry with the Hurwitz integral quaternions, $\overline{D}_4$ can be rotated to contain $\mathbb{Z}^4$.
Lagrange's four square theorem then implies $N(D_4)=\{k^2:k\in\mathbb{N}\}$.
Next, $\|z\|_2^2\in\mathbb{N}$ for every nonzero $z\in\overline{E}_8$, and furthermore, the theta function for $E_8$ gives
\[
|\{z\in\overline{E}_8:\|z\|_2^2=k\}|
=240\sum_{d|k}d^3
>0
\]
for each $k\in\mathbb{N}$; see p.~122 of~\cite{ConwayS:13}.
As such, $N(E_8)=\{k^4:k\in\mathbb{N}\}$ in this case.
Similarly, the theta function for $\Lambda_{24}$ gives that $N(\Lambda_{24})=\{k^{12}:k\in\mathbb{N}\}$.
We suspect that Lemma~\ref{lem.dilated lattices} also holds for $L=\Lambda_{24}$, but a proof requires a different idea since, by Hurwitz's theorem, $\{\mathbb{R},\mathbb{C},\mathbb{H},\mathbb{O}\}$ accounts for all Euclidean Hurwitz algebras. 
In Example~3.2(d) of~\cite{BarnesS:83}, Barnes and Sloane give $T\in\operatorname{CO}(24)$ such that $T\Lambda_{24}\geq \Lambda_{24}$ and $|T\Lambda_{24}: \Lambda_{24}|=2^{12}$, i.e., $(T\Lambda_{24})/\Lambda_{24}$ is a lattice code of size $2^{12}$.

Recall that the covering radius $R$ of a lattice $L\subseteq\mathbb{R}^m$ is the supremum of all $r$ for which $\mathbb{R}^m\setminus\bigcup_{x\in L}B(x,r)$ is nonempty, and a point $z\in\mathbb{R}^m$ is known as a \textbf{deep hole} of $L$ if $\min\{\|z-x\|_2:x\in L\}=R$; see~\cite{ConwayS:13}, for example.
Consider codes $\{s+L:s\in S\}$ such that for each $s,t\in S$ with $s\neq t$, it holds that $s-t$ is a deep hole of $L$.
We refer to such codes as \textbf{holy codes}.

\begin{lemma}
\label{lem.holy characterization}
Suppose $L\subseteq\mathbb{R}^m$ is a lattice with covering radius $R$, and let $C\subseteq\mathbb{R}^m/L$ be finite.
The following are equivalent:
\begin{itemize}
\item[(i)]
$C$ is holy.
\item[(ii)]
For every $x,y\in C$ with $x\neq y$, it holds that $d(x,y)=R$.
\item[(iii)]
$\min\{d(x,y):x,y\in C,x\neq y\}=R$.
\end{itemize}
Furthermore,
\begin{itemize}
\item[(a)]
every holy code in $\mathbb{R}^m/L$ is optimal, and
\item[(b)]
if there exists a holy code of size $n$ in $\mathbb{R}^m/L$, then every optimal code of size $n$ in $\mathbb{R}^m/L$ is a holy code.
\end{itemize}
\end{lemma}

\begin{proof}
First, we observe that for every $s,t\in\mathbb{R}^m$, it holds that
\begin{align*}
d(s+L,t+L)
&=\min\{\|(s+u)-(t+v)\|_2:u,v\in L\}\\
&=\min\{\|(s-t)-w\|_2:w\in L\}\\
&\leq R,
\end{align*}
where the last step applies the definition of covering radius.
Furthermore, equality occurs precisely when $s-t$ is a deep hole.
We will repeatedly use both of these facts.

For (i)$\Rightarrow$(ii), suppose $C=\{s+L:s\in S\}$ is holy.
Then for every $s,t\in S$ with $s\neq t$, it holds that $s-t$ is a deep hole, and so $d(s+L,t+L)=R$ by our intermediate result.
Next, (ii)$\Rightarrow$(iii) is immediate.
For (iii)$\Rightarrow$(i), take any $s,t\in S$ with $s\neq t$.
Then (iii) and our intermediate result together give $R\leq d(s+L,t+L)\leq R$, in which case the second part of our intermediate result gives that $s-t$ is a deep hole.
Since $s$ and $t$ were chosen arbitrarily, it follows that $C$ is holy.

For (a), first note that by our intermediate result, the optimal codes of size $2$ are precisely those of the form $\{s+L,t+L\}$ where $s-t$ is a deep hole, and the minimum distance equals $R$.
This proves (a) in the case of holy codes of size $2$.
For larger holy codes, (i)$\Rightarrow$(iii) gives that the minimum distance is $R$, and by the monotonicity of the maximum minimum distance, it follows that these codes are optimal.
For (b), suppose there exists a holy code of size $n$.
Then the minimum distance equals $R$, and by (a), this is the maximum minimum distance.
As such, any other optimal code must also have minimum distance $R$, and (iii)$\Rightarrow$(i) gives that such codes are necessarily holy.
\end{proof}

We can use holy codes to characterize small optimal codes:

\begin{lemma}
\label{lem.lattice holy correspondence}
Select $L\in\{A_1,A_2,D_4,E_8,\Lambda_{24}\}$ and put $m:=\operatorname{dim}L$ and
\[
n_L:=\min( N(L)\setminus\{1\} ).
\]
Then for each $n\in\{2,\ldots,n_L\}$, the optimal codes in $\mathbb{R}^m/L$ of size $n$ are holy.
\end{lemma}

\begin{proof}
Consider any code $(TL)/L$ of size $n_L$ that is constructed in the proof of Lemma~\ref{lem.dilated lattices} (or in Example~3.2(d) of~\cite{BarnesS:83} in the case $L=\Lambda_{24}$).
This code has minimum distance $\ell(L)\cdot n_L^{-1/m}$, which in each case equals the covering radius of $L$~\cite{ConwayS:13}.
The result then follows from Lemma~\ref{lem.holy characterization}.
\end{proof}

\subsubsection{Testing conjectures}

By virtue of their translation symmetry, Conjecture~\ref{conj.1} necessarily holds for any unicorn sequence of lattice codes.
(Of course, Conjecture~\ref{conj.1} remains unproven even for the flat torus $\mathbb{R}^2/A_2$ since there may be unicorn sequences consisting of non-lattice codes.)
In order to obtain such a unicorn sequence, we must avoid nonuniqueness as in Figure~\ref{fig.lattice}.
Let $N':=\{1, 3, 4, 9, 12, 16, 25, \ldots\}$ denote the positive integers for which (i) no prime factor is $1\bmod3$ and (ii) every prime factor that is $2\bmod3$ has even multiplicity (OEIS A230781).
These are precisely the squared radii of circles centered at the origin that intersect $\overline{A}_2$ at exactly $6$ points; see p.~112 of~\cite{ConwayS:13}.
This sequence yields the following:

\begin{lemma}
\label{lem.triangular torus unicorn sequence}
For every $n\in N'$, the optimal code in $\mathbb{R}^2/A_2$ of size $n$ is unique up to isometry.
\end{lemma}

\begin{proof}
Consider any optimal code $C\subseteq\mathbb{R}^2/A_2$ of size $n\in N'$, and fix $z_0\in \overline{A}_2$ with $\|z_0\|_2^2=n$.
Since $N'\subseteq N(A_2)$, there exists a lattice code in $\mathbb{R}^2/A_2$ of size $n$ by Lemma~\ref{lem.dilated lattices}, and so $C$ must be a lattice code up to isometry by Lemma~\ref{eq.lattice codes opt}.
In particular, there exists $T\in\operatorname{CO}(2)$ such that $C=(TA_2)/A_2$ up to isometry.
Let $R$ denote reflection about the $x$-axis.
Since $RA_2=A_2$, we also have $C=(TRA_2)/A_2$ up to isometry, meaning $T$ has positive determinant without loss of generality.
Put $z:=T^{-1}e_1\in \overline{A}_2$ and recall from the proof of Lemma~\ref{lem.dilated lattices} that $\|z\|_2^2=|TA_2:A_2|=n$.
Then $z$ must reside in the orbit $\{U^kz_0:k\in\{0,\ldots,5\}\}$, where $U$ denotes rotation by $\pi/3$ radians.
Take $T_0$ to be the unique member of $\operatorname{CO}(2)$ with positive determinant that maps $z_0$ to $e_1$.
Then $T_0=TU^k$ for some $k$, and since $T_0A_2=TU^kA_2=TA_2$ (by the symmetry of $A_2$), it follows that $C=(T_0A_2)/A_2$ up to isometry.
\end{proof}

We suspect that $\mathbb{R}^m/L$ is also a unicorn space for every $L\in\{D_4,E_8,\Lambda_{24}\}$, but we do not have a proof.
The primary obstacle is establishing uniqueness, which we accomplish for both $D_4$ and $E_8$ in the smallest nontrivial case below.
We note that a portion of this result is implied by the conjecture that $D_4$ gives the unique optimal periodic sphere packing in $\mathbb{R}^4$, and so one might treat this as evidence in favor of that conjecture:

\begin{theorem}
\label{thm.computational lattice result}
For each $L\in\{A_1,A_2,D_4,E_8\}$, the lattice code in $\mathbb{R}^m/L$ of size $n_L$ is optimal and unique up to isometry, and every optimal code of size $n_L-1$ is a subset of this lattice code up to isometry.
\end{theorem}

Our proof of this result relies on computer assistance.
The code is included as an ancillary file in the arXiv version of the paper.

\begin{proof}[Proof of Theorem~\ref{thm.computational lattice result}]
First, we consider optimality.
For all cases but $D_4$, optimality follows from Lemma~\ref{eq.lattice codes opt}, though this requires the fact that these lattices give optimal sphere packings.
Instead, the proof of Lemma~\ref{lem.lattice holy correspondence} provides a direct proof of optimality that also treats the $D_4$ case.

Next, we consider uniqueness.
For $A_1$, uniqueness follows from the pigeonhole principle, while for $A_2$, uniqueness follows from Lemma~\ref{lem.triangular torus unicorn sequence}.
For $D_4$, we know from Lemma~\ref{lem.lattice holy correspondence} that every optimal code $C$ of size $n_L$ is holy.
Without loss of generality, $C$ takes the form $\{L,x+L,y+L,z+L\}$, where each $x$, $y$ and $z$ is a deep hole.
For $D_4$, the deep holes of minimum norm form the vertices of the $24$-cell.
Explicitly, they are all signed permutations of $(1,0,0,0)$ and $(\frac{1}{2},\frac{1}{2},\frac{1}{2},\frac{1}{2})$.
Reducing these deep holes modulo $D_4$ produces only $3$ cosets of deep holes.
It follows that the optimal code is unique up to isometry.
For $E_8$, we also apply Lemma~\ref{lem.lattice holy correspondence}.
The deep holes of minimum norm are precisely the members of $E_8$ of norm $2$ scaled by $1/2$.
Explicitly, there are $2160$ deep holes of minimum norm, and they are obtained by collecting permutations of the following vectors in $\mathbb{R}^8$: $(1,0_7)$ with both sign patterns, $\frac{1}{2}\cdot(1_4,0_4)$ with all $2^4$ sign patterns, and $\frac{1}{4}\cdot(3,1_7)$ with an odd number of minus signs.
Reducing modulo $E_8$ produces $135$ cosets of deep holes.
Consider the graph $G(E_8)$ whose $136$ vertices are $L$ and the $135$ cosets of deep holes, and for which two vertices are adjacent if their difference is a coset of deep holes.
The clique number of this graph is $16$, implying the largest holy code in $\mathbb{R}^8/E_8$ has size $16$.
There are $270$ cliques of size $16$.
Next, we consider how the isometry group $\operatorname{Aut}(E_8)\ltimes(\mathbb{R}^8/E_8)$ acts on these cliques.
Each clique is invariant under translations by members of the clique, so it suffices to consider the action of $\operatorname{Aut}(E_8)$.
To accomplish this, we leverage the fact that $\operatorname{Aut}(E_8)$ is generated by the reflections about the hyperplanes orthogonal to the $240$ nonzero vectors of minimum norm in $E_8$~\cite{ConwayS:13}.
In particular, we start with an arbitrary clique of the graph, and then we iteratively apply a random reflection to reach another clique.
The result is a random walk along equivalent cliques that rapidly visits all cliques.
It follows that the optimal code of size $16$ in $\mathbb{R}^8/E_8$ is unique up to isometry.

Finally, we consider the optimal codes of size $n_L-1$.
By Lemma~\ref{lem.lattice holy correspondence}, these codes are necessarily holy.
If we generalize the above definition of $G(E_8)$ to $G(L)$, it suffices to show that every clique of size $n_L-1$ is contained in a clique of size $n_L$.
This easily holds in the cases $G(A_1)=K_2$, $G(A_2)=K_3$ and $G(D_4)=K_4$, and a quick calculation reveals that it also holds for $G(E_8)$.
\end{proof}

The Leech lattice $\Lambda_{24}$ is notably absent from Theorem~\ref{thm.computational lattice result}, and the reason is simple:
our computational approach to verify uniqueness requires access to all deep holes of minimum norm, but there are over $10^{19}$ such points in the case of the Leech lattice~\cite{SikiricSV:10}.
As such, a different approach is required to treat this case.

Next, to evaluate Conjecture~\ref{conj.2}, we first identify the symmetry strength.
We claim that for each $L\in\{A_1,A_2,D_4,E_8,\Lambda_{24}\}$, the symmetry strength of $\mathbb{R}^m/L$ is $2$.
(Note that we already know that this holds for $\mathbb{R}^1/A_1$, since this space is isometric to the circle.)
First, for every $s+L\in \mathbb{R}^m/L$, it holds that any nontrivial rotation about $s$ fixes $s$, and if the rotation part of this affine rotation belongs to $\operatorname{Aut}(L)$, then it also fixes $s+L$.
Since $\operatorname{Aut}(L)$ is nontrivial, it follows that the symmetry strength is at least $1$.
Next, for every $s+L,t+L\in \mathbb{R}^m/L$, it holds that $\{s+L,t+L\}$ is invariant under the action of $g\in\operatorname{Aut}(L)$ defined by $g(x)=s+t-x$, and so the symmetry strength is at least $2$.
Finally, for a generic choice of $s+L,t+L\in \mathbb{R}^m/L$, it holds that the triangle with vertices at $T:=\{L,s+L,t+L\}$ is scalene, and so $gT=T$ only if $gL=L$, $g(s+L)=s+L$, and $g(t+L)=t+L$.
Since $gL=L$, it follows that $g\in\operatorname{Aut}(L)\leq\operatorname{Aut}(L)\ltimes(\mathbb{R}^m/L)$.
Since $s+L$ is generic, it holds that $s$ is the unique point of its norm in $s+L$, and so $g(s+L)=s+L$ implies $g(s)=s$.
However, since $s$ is generic and $\operatorname{Aut}(L)$ is finite, it follows that $g=\operatorname{id}$.
Overall, the symmetry strength of $\mathbb{R}^m/L$ is exactly $2$, as claimed.

In \cite{DickensonGKX:11,ConnellyD:14}, it is conjectured that for every $n\in N(A_2)$ with $n>1$ for which $n-1\not\in N(A_2)$, the optimal codes of size $n-1$ are obtained by removing any point from any lattice code of size $n$.
This is known to hold for $n\in\{3,7\}$ and is open for $n=9$.
The conjecture implies that Conjecture~\ref{conj.2} holds for any unicorn sequence of lattice codes in $\mathbb{R}^2/A_2$, e.g., the sequence identified in Lemma~\ref{lem.triangular torus unicorn sequence}.
Presumably, for each $L\in\{D_4,E_8,\Lambda_{24}\}$, it similarly holds that for every $n\in N(L)$ with $n>1$, the optimal codes of size $n-1$ in $\mathbb{R}^m/L$ are obtained by removing a point from any lattice code of size $n$.
Theorem~\ref{thm.computational lattice result} establishes this phenomenon in the smallest case for both $D_4$ and $E_8$.

\subsection{Metric graphs}

Let $G=(V,E)$ be a graph with positive edge weights $\{w_e\}_{e\in E}$.
Here, we allow loops and multiple edges.
For each edge $e\in E$, consider an interval of length $w_e$, and glue the endpoints of these intervals according to the incidence rule prescribed by $G$.
The resulting set $M$, known as a \textbf{metric graph}, enjoys a notion of geodesic distance.
We restrict our attention to finite graphs since they determine compact metric graphs.
Notice that by smoothing vertices as necessary, every non-cyclic metric graph can be described in terms of a graph in which no vertex has degree $2$.
This identification allows us to identify the isometries of metric graphs:

\begin{lemma}
Consider a connected graph $G=(V,E)$ with loops $L\subseteq E$ and with positive edge weights $\{w_e\}_{e\in E}$.
Let $M$ denote the corresponding metric graph.
\begin{itemize}
\item[(a)]
If $G$ is a cycle, then the isometry group of $M$ is isomorphic to orthogonal group $\operatorname{O}(2)$.
\item[(b)]
If $G$ has no vertices of degree $2$, then the isometry group of $M$ is isomorphic to the direct product between $(\mathbb{Z}/2\mathbb{Z})^L$ and the subgroup of automorphisms of $G$ that hold the edge weights $\{w_e\}_{e\in E}$ invariant.
\end{itemize}
\end{lemma}

\begin{proof}[Proof sketch]
First, (a) follows from the fact that $M$ is isometric to a circle.
It remains to prove (b).
Fix an orientation for each loop in $M$.
For each $e\in E$, let $u_e\in M$ denote the midpoint of $e$.
Put $U:=\{u_e:e\in E\}$.
Consider any isometry $\psi\colon M\to M$.
Then $\psi$ permutes the members of $V$, as well as the members of $U$.
Furthermore, since we may identify the edge $e$ with the ball about $u_e$ of radius $w_e/2$, it follows that $\psi$ permutes the members of $E$, while possibly flipping any subset of the loops.
In this way, $\psi$ determines a member of $(\mathbb{Z}/2\mathbb{Z})^L$ and an automorphism of $G$ that holds the edge weights invariant.
For the reverse direction, select $b\in(\mathbb{Z}/2\mathbb{Z})^L$ and let $\phi$ denote an automorphism of $G$ that holds the edge weights invariant.
We use $(b,\phi)$ to construct a mapping $\psi\colon M\to M$.
For every $e\in L$ and every $t\in[0,w_e)$, let $p(e,t)\in M$ denote the unique point on $e$ at distance $t$ from $v\in e$ in the direction determined by $b(e)$, and define $\psi(p(e,t)):=p(\phi(e),t)$.
For every $v\in e\in E\setminus L$ and every $t\in[0,w_e/2]$, let $p(v,e,t)\in M$ denote the unique point on $e$ at distance $t$ from $v$, and define $\psi(p(v,e,t)):=p(\phi(v),\phi(e),t)$.
One may show that $\psi$ is an isometry of $M$, as desired.
\end{proof}

\subsubsection{Uniquely optimal codes}

We start by focusing on unit-distance metric graphs, which have the property that each edge has unit length.
Such metric graphs frequently enjoy one of two ``obvious'' families of optimal codes: either uniformly distribute $k$ points on each edge, or do so with an additional point at each vertex.
Explicitly, for the first code, which we denote by $C_k$, we distribute $k$ points on each edge at locations $\frac{1}{2k},\frac{3}{2k},\ldots, \frac{2k-1}{2k}$ in a linear parameterization of the edge over $[0,1]$.
For the second code, which we denote by $C'_k$, we put a point on each vertex and distribute $k$ points on each edge at locations $\frac{1}{k+1},\frac{2}{k+1},\ldots, \frac{k}{k+1}$.
Then $C_k$ has size $k|E|$ and minimum distance $1/k$, while $C'_k$ has size $|V|+k|E|$ and minimum distance $1/(k+1)$.
The following characterizes when these codes are uniquely optimal:

\begin{theorem}
\label{thm.unique codes in unit-distance metric graphs}
Consider a connected unit-distance metric graph $G=(V,E)$ and select $k\in\mathbb{N}$.
\begin{itemize}
\item[(a)]
The code $C'_k$ is optimal and unique up to isometry if and only if $G$ is a tree.
\item[(b)]
The code $C_k$ is optimal and unique up to isometry if and only if $\displaystyle\min_{v\in V}\operatorname{deg}(v)\geq 2$.
\end{itemize}
\end{theorem}

\begin{proof}
First, we show that the minimum distance of any code of size $|V|+k|E|$ is at most $1/(k+1)$.
This certainly holds if $|E|=1$.
Now select $m\in\mathbb{N}$ with $m\geq 2$ and suppose our claim holds for every unit-distance metric graph with $|E|<m$.
Consider any unit-distance metric graph $G$ with $|E|=m$ and any code $C$ of size $|V|+k|E|$ in $G$.
Suppose the interior of any edge $e\in E$ contains $k$ or fewer members of $C$, and let $C'$ denote these points.
Then applying the induction hypothesis to the connected components of $G-e$ gives $\delta(C)\leq \min(\delta(C'),\delta(C\setminus C'))\leq 1/(k+1)$.
It remains to consider the case in which the interior of every $e\in E$ contains at least $k+1$ members of $C$.
Then $(k+1)|E|
\leq|C|
=|V|+k|E|$,
and so $|E|\leq|V|$.
By connectivity, we must have either $|E|=|V|$ or $|E|=|V|-1$.
In the former case, $G$ contains a cycle, in which case the average distance between consecutive members of $C$ in this cycle is at most $1/(k+1)$, and so $\delta(C)\leq 1/(k+1)$.
In the latter case, $G$ is a tree.
Then $|C|=|V|+k|E|=(k+1)|E|+1$, and so one of the leaf edges contains exactly $k+1$ points in its interior and has no point at the corresponding leaf vertex $y$.
Take $x$ to be the member of $C$ that is closest to $y$ and consider the modification $C':=(C\setminus\{x\})\cup\{y\}$.
Then $C'$ has only $k$ points in the interior of the edge incident to $y$, and so the induction hypothesis again gives $\delta(C)\leq \delta(C')\leq 1/(k+1)$.

Overall, the minimum distance of any code of size $|V|+k|E|$ is at most $1/(k+1)$.
This bound can be saturated by putting a point at each vertex and uniformly distributing $k$ points on each edge.
As such, the optimal codes of size $|V|+k|E|$ have minimum distance exactly $1/(k+1)$.

\medskip

(a) For ($\Rightarrow$), consider the contrapositive and suppose $|E|\geq |V|$.
Then one may distribute $k+1$ points on each edge at locations $\frac{1}{2(k+1)},\frac{3}{2(k+1)},\ldots, \frac{2(k+1)-1}{2(k+1)}$ in a linear parameterization of the edge over $[0,1]$, and then take any subset of size $|V|+k|E|$ to produce an inequivalent code of minimum distance $1/(k+1)$.
For ($\Leftarrow$), we induct on the size of the tree.
The claim holds for the tree on $2$ vertices, as the tree in this case reduces to an interval.
Now suppose it holds for all trees on $n\geq 2$ vertices, and consider any tree $M$ on $n+1$ vertices and any optimal code $C$ on $M$.
This tree has a leaf edge.
Delete this edge (and the corresponding leaf vertex) to produce a smaller tree $M'$, and consider the code $C':=C\cap M'$.
Since $\delta(C')\geq1/(k+1)$, we know that $|C'|\leq|C|-(k+1)$ by the induction hypothesis.
Also, since $\delta(C)=1/(k+1)$, we know that $|C\setminus C'|\leq k+1$ by the pigeonhole principle.
Combining these inequalities then gives $|C\setminus C'|= k+1$.
By the induction hypothesis, the desired code $C'_k\subseteq M'$ is uniquely optimal for $M'$, and this extends uniquely to the claimed optimal code for $M$.

\medskip

For (b), we start with ($\Rightarrow$).
Consider the contrapositive and take $G$ with minimum degree $1$.
Then $G$ has a leaf vertex $y$.
Take $x$ to be the member of $C_k$ that is closest to $y$ and consider the modification $C':=(C_k\setminus\{x\})\cup\{y\}$.
Then $\delta(C_k)\leq\delta(C')$, but $C'$ is inequivalent to $C_k$.
Next, for ($\Leftarrow$), first note that if $\min_{v\in V}\operatorname{deg}(v)\geq 2$, then $G$ is not a tree, and so $|E|\geq|V|$, 
which in turn implies $|C_k|=k|E|\geq|V|+(k-1)|E|$.  As proven above, any code of size $|V|+(k-1)|E|$ has minimum distance at most $1/k=\delta(C_k)$; thus, monotonicity of maximum minimum distance implies that $C_k$ is optimal.
It remains to show that $C_k$ is unique up to isometry.
To this end, suppose that $C$ is a code of size $k|E|$ with minimum distance $1/k$.

\textbf{Case I:}
$C$ intersects $V$.
Put $U:=C\cap V\neq\varnothing$, and let $E'$ denote the set of edges with at least one endpoint in $U$.
By pigeonhole, each edge has at most $k+1$ points in its closure, and if it has exactly this many points, then both endpoints are in $C$.
For each $e\in E$, let $c(e)$ denote the number of points in the interior of $e$.
Then
\[
k|E|
=|C|
=|U|+\sum_{e\in E}c(e)
\leq |U|+k|E\setminus E'|+(k-1)|E'|,
\]
and so $|E'|\leq|U|$.
Next, we apply the fact that the minimum degree of $G$ is at least $2$:
\[
|E'|
\leq|U|
=\sum_{u\in U}\frac{1}{\operatorname{deg}(u)}\sum_{e\in E}\left\{\begin{array}{cl}1&\text{if }u\in e\\0&\text{otherwise}\end{array}\right\}
\leq\frac{1}{2}\sum_{e\in E}\sum_{u\in U}\left\{\begin{array}{cl}1&\text{if }u\in e\\0&\text{otherwise}\end{array}\right\}
\leq|E'|.
\]
As such, equality is forced.
Note that equality in the second inequality requires $\operatorname{deg}(u)=2$ for every $u \in U$, while equality in the last inequality implies that both endpoints of every $e\in E'$ reside in $U$.
It follows that $U$ induces a cyclic connected component, and since $G$ is connected by assumption, $G$ must be a cycle.
Then $C_k$ is unique up to isometry, although the isometry group in this case is $\operatorname{O}(2)$; see Section~\ref{sec.circle}.

\textbf{Case II:}
$C$ does not intersect $V$.
Then each edge has at most $k$ points on its interior, and so by pigeonhole, every edge has exactly $k$ points on its interior. We will prove the claim by measuring the lengths along all of the edges in two different ways. First, note that for edge $e$ with vertices $u$ and $v$, the distance between the elements of $C$ closest to $u$ and $v$, respectively, must be at least $(k-1)/k$. Second, for vertex $v$, the sum of lengths from $v$ to closest point in $C \cap e$ for each $e$ incident with $v$ is at least $\deg(v)/2k$.  If $\deg(v) >2$, this inequality is saturated precisely when each of those distances is $1/(2k)$. Thus, measuring the graph edge lengths, we have via the handshaking lemma that
\[
\abs{E} \geq \sum_{v \in V} \frac{\deg(v)}{2k} + \sum_{e \in E} \frac{k-1}{k} = \frac{2\abs{E}}{2k} +\frac{k-1}{k}\abs{E} = \abs{E}.
\]
Hence, at every edge $e$ incident to a vertex of degree strictly greater than $2$, $e \cap C = e \cap C_k$. If there is a vertex $v$ of degree $2$, then either all vertices in $G$ have degree $2$ and we have a cycle graph where $C_k$ is unique up to isomorphism or $v$ lies on a path $P$ with end vertices $u$ and $w$ of degree $> 2$ in $G$. The points in $P \cap C$ closest to $u$ and $v$ must be at least $1/2k$ from $u$ and $v$, respectively. Thus, pigeonhole delivers that $P \cap C = P \cap C_k$.  I.e., $G$ is either a cycle and $C$ is isomorphic to $C_k$ under $O(2)$ or $G$ has at least one vertex of degree $>2$ and $C = C_k$.
\end{proof}

Theorem~\ref{thm.unique codes in unit-distance metric graphs} offers two different types of unicorn unit-distance metric graphs: trees and connected graphs with no leaves.
While the latter case is difficult to analyze, the following result generalizes the former case beyond the setting of unit-distance graphs.
Specifically, we refer to any metric graph that arises from a tree as a metric tree.
Without loss of generality, no vertex has degree $2$, and in this case, we refer to the non-leaf vertices as junctions.
In the following result, we use ``divisor'' in the sense defined in Section~\ref{sec.orthotopes}.

\begin{theorem}
\label{thm.unicorn metric trees}
Consider any metric tree $(M,d)$, let $\ell$ denote the length of the shortest edge in $M$, and let $J$ and $L$ denote the sets of junctions and leaves in $M$, respectively.
The following are equivalent for $\delta\in(0,\ell)$:
\begin{itemize}
\item[(i)]
The optimal code of minimum distance $\delta$ in $(M,d)$ is unique up to isometry.
\item[(ii)]
The optimal code of minimum distance $\delta$ in $(M,d)$ is unique.
\item[(iii)]
There exist functions $f\colon J\to L$ and $g\colon J\to 2^L$ such that
\begin{itemize}
\item[(a)]
$f(u)\in g(u)$ for every $u\in J$,
\item[(b)]
$|g(u)|=\operatorname{deg}(u)$ for every $u\in J$,
\item[(c)]
for each $u\in J$, the paths from $u$ to each $v\in g(u)$ are internally disjoint,
\item[(d)]
$\delta$ is a divisor of $d(u,f(u))+d(u,v)$ for every $u\in J$ and $v\in g(u)\setminus\{f(u)\}$,
\item[(e)]
$\operatorname{frac}(\tfrac{d(u,f(u))}{\delta})\leq\tfrac{1}{2}$ for every $u\in J$.
\end{itemize}
\end{itemize}
Furthermore, one such $\delta$ exists if and only if $(M,d)$ is a unicorn space.
\end{theorem}

\begin{proof}
Since the isometry group of a metric tree is finite, (i) implies that the optimal code of minimum distance $\delta$ exhibits no rattle.
As such, this code necessarily contains all leaves, so we may start building the code from these leaf vertices and iteratively add points of distance $\delta$ until reaching a junction at each leaf edge.
Next, consider the tree $M_1$ obtained by deleting leaf edges from $M_0:=M$.
Then the code locations on the leaf edges of $M_1$ are determined by the code locations on the leaf edges of $M_0$. Repeat with $M_2$, etc.
The code is completely determined in this way, and so we conclude (ii).
The converse (ii)$\Rightarrow$(i) is immediate.

Next, supposing (ii) holds, we build functions $f$ and $g$ as described in (iii). 
We will use throughout this argument the fact that $C$ has no rattle.
Given the optimal code $C$ of minimum distance $\delta$, then for each $u\in J$, select an edge $e$ in $M$ incident to $u$ for which there exists $x\in C\cap e$ such that $d(x,u)\leq \delta/2$; in the case where $u\in C$, select $e$ incident to $u$ arbitrarily.
Traveling away from $u$ along $e$, we find a member of $C$ after every $\delta$ units of distance.
When we arrive at a junction, there must be another edge leaving the junction that has a member of $C$ that is $\delta$ away from our sequence.
We continue in this way until reaching a leaf $v\in L$ and we put $f(u):=v$.
For each of the other edges incident to $u$, we can follow the same process to identify $\operatorname{deg}(u)-1$ other leaves, which we combine with $f(u)$ to produce $g(u)$.
Note that by construction, $(f,g)$ satisfies (a)--(c).
For what follows, it is convenient to think of any $u\in C\cap J$ as residing in the edge emanating from $u$ towards $f(u)$ and not on any other edge incident to $u$.
If $\eps$ is the distance from $u$ to the closest member of $C$ to $u$ that resides in the path from $u$ to $f(u)$, then $\delta - \eps$ is the distance from $u$ to each of the closest members of $C$ to $u$ that resides in the path from $u$ to any $v\in g(u)\setminus\{f(u)\}$. Then (d) follows from the fact that $d(u,f(u))+d(u,v)$ is $\eps$ + $\delta -\eps$ plus an integer number of $\delta$ ``steps.''
Furthermore, since any two of these closest members  $w,w'\in C$ to $u$ along these paths must satisfy the constraint $d(w,w')\geq\delta$, it follows that (e) necessarily holds.

Now suppose (iii) holds.
Then (a), (c), and (d) together ensure that for every $u\in J$, one may place members of $C$ uniformly along the path from $f(u)$ to each $v\in g(u)\setminus\{f(u)\}$.
For each $v\in L$, it holds that $v\in g(u)$, where $u\in J$ is the neighbor of $v$.
As such, this determines all members of $C$ along edges emanating from leaves.
Furthermore, (b) and (c) together ensure that this determines all members of $C$ along edges emanating from junctions.
Overall, $C$ is completely determined, and (e) ensures that no two points are closer than $\delta$ apart.
We may therefore conclude (ii).

For the last claim, the ($\Leftarrow$) direction is immediate: $(M,d)$ is a unicorn space only if there is a decreasing sequence of $\delta$ for which (i) holds.
For the ($\Rightarrow$) direction, we fix $f$ and $g$ for which (iii) holds for some $\delta=\delta_0\in(0,\ell)$.
It suffices to show that there are infinitely many $\delta$ that simultaneously satisfy both (d) and (e) for this $f$ and $g$.
Let $\Delta$ denote the greatest common divisor of $d(u,f(u))+d(u,v)$ over all $u\in J$ and $v\in g(u)\setminus\{f(u)\}$.
Then any other positive common divisor takes the form $\delta=\Delta/k$ for some $k\in\mathbb{N}$.
Put $z_u:=d(u,f(u))/\Delta$ for each $u\in J$.
Then it suffices to show that there are infinitely many $k\in\mathbb{N}$ for which $\operatorname{frac}(kz_u)\leq1/2$ for every $u\in J$.
Note that for $\mathbb{T}:=\mathbb{R}/\mathbb{Z}$ and any $y \in \mathbb{R}$, $y + \mathbb{Z} = \operatorname{frac}(y) + \mathbb{Z}  \in \mathbb{T}$.  Thus, it is helpful to define $\alpha\in\mathbb{T}^J$ whose $u$-coordinate is the coset $z_u+\mathbb{Z}$.
Then we seek an infinite sequence of $k\in\mathbb{N}$ for which $x(k):=k\alpha\in[0,\tfrac{1}{2}]^J\subseteq\mathbb{T}^J$.
Write $k_0:=\Delta/\delta_0$.
Then $x(k_0)\in [0,\tfrac{1}{2}]^J$.
Let $B\subseteq J$ denote the indices of the entries of $x(k_0)$ that reside in $(\tfrac{1}{2}\mathbb{Z})/\mathbb{Z}$, and put $A:=J\setminus B$.
Then projecting onto $A$ and $B$ gives the decomposition $x=x_A+x_B$, where $x_B$ is a periodic function.
If $B=J$, we are done.
Otherwise, by Proposition~1.1.5 in~\cite{Tao:12}, it holds that $x_A$ is asymptotically equidistributed in a union $U$ of finitely many cosets of a subtorus $T'\subseteq \mathbb{T}^J$, and due to the periodicity of $x_B$, it further holds that $T'$ is a subtorus of $\mathbb{T}^A$.
Considering $x_A(k_0)\in(0,\tfrac{1}{2})^A\times\{0\}^B$, there exists a small neighborhood $N$ of $0$ in $T'$ such that $x(k_0)+z\in U\cap [0,\tfrac{1}{2}]^J$ for every $z\in N$.
It follows that $\mu(U\cap[0,\tfrac{1}{2}]^J)\geq\mu(N)>0$, where $\mu$ denotes the Haar measure on $U$.
The result then follows from asymptotic equidistribution, since an asymptotic proportion $\mu(U\cap[0,\tfrac{1}{2}]^J)$ of $k\in\mathbb{N}$ satisfies $x(k)\in[0,\tfrac{1}{2}]^J$.
\end{proof}

Consider any unicorn metric tree $(M,d)$ with a nontrivial isometry $g$, and select $\delta\in(0,\ell)$ for which there exists an optimal code $C$ of minimum distance $\delta$ in $(M,d)$ that is unique up to isometry.
Then $gC$ is also an optimal code.
However, Theorem~\ref{thm.unicorn metric trees} gives that $C$ is the unique optimal code, and so it must hold that $gC=C$.
This establishes Conjecture~\ref{conj.1} for metric trees.

\subsubsection{Linear programming bound}
\label{sec.lp}

Consider any partition $|C|=\sum_{e\in E}t_e$ with $t_e\in\mathbb{N}\cup\{0\}$, and suppose that for each $e\in E$, we assign $t_e$ members of $C$ to $e$.
(In particular, we allow a member of $C$ to be an endpoint of $e$, meaning there are different choices of assignments $\{t_e\}_{e\in E}$ that can be used to describe such codes.)
For each $v\in e\in E$ with $t_e>0$, let $d_{ve}\geq0$ denote the distance between $v$ and the closest member of $C$ to $v$ that was assigned to $e$.
Then the optimal codes for the assignment $\{t_e\}_{e\in E}$ can be obtained by solving the following linear program in terms of the edge lengths $\{w_e\}_{e\in E}$.
(For notational simplicity, we write this program assuming the graph is simple, which it is without loss of generality by virtue of subdivision.)
\begin{alignat*}{3}
\text{maximize}\qquad && \delta \\
\text{subject to}\qquad && d_{ve}+(t_e-1)\delta + d_{v'e}&\leq w_e &&\qquad e=\{v,v'\}\in E,~t_e>1\\
&& d_{ve} + d_{v'e}&= w_e &&\qquad e=\{v,v'\}\in E,~t_e=1\\
&& d_{ve}+d_{ve'}&\geq \delta &&\qquad e,e'\in E,~ e\neq e', ~ v\in e\cap e', ~ t_e,t_{e'}>0\\
&& d_{ve}&\geq0 &&\qquad v\in e\in E, ~ t_e,t_{e'}>0
\end{alignat*}
Specifically, this linear program is a relaxation of the optimal code problem that is necessarily tight when $\delta(C)\leq\min_{e\in E}w_e$ for every $C\subseteq M$ with $|C|=\sum_{e\in E}t_e$.
In particular, for every maximizer $\{d_{ve}\}_{v\in e\in E,t_e>0}$, the corresponding optimal codes have a point on $e$ at distance $d_{ve}$ away from $v$ for each $v\in e\in E$ with $t_e>0$.
If $t_e\geq 2$, then the other $t_e-2$ points assigned to $e$ are distributed in any way such that neighboring points have distance at least $\delta$.
In the case of equality $d_{ve}+(t_e-1)\delta + d_{v'e}=w_e$, these $t_e-2$ points are necessarily uniformly distributed between the extreme code points on $e$.

For each metric graph $M$, every sufficiently large code $C\subseteq M$ satisfies $\delta(C)\leq\min_{e\in E}w_e$, meaning the above linear program is tight in these settings.
In particular, we claim that it suffices to have
\begin{equation}
\label{eq.code bound size}
|C|
\geq 2\cdot\frac{\sum_{e\in E}w_e}{\min_{e\in E}w_e}.
\end{equation}
To see this, for each $x\in C$, select a closest $y\in C$ and consider the interval of length $\delta(C)/2$ in $M$ that starts at $x$ and traverses toward $y$.
These intervals are necessarily internally disjoint, and so
\[
|C|\cdot\frac{\delta(C)}{2}
\leq \sum_{e\in E}w_e
\leq \frac{|C|}{2}\cdot\min_{e\in E}w_e.
\]
Rearranging gives $\delta(C)\leq\min_{e\in E}w_e$, as desired.
Thus, by testing all possible assignments $\{t_e\}_{e\in E}$, the above linear program characterizes all sufficiently large optimal codes.
The number of such assignments is $\binom{n+|E|-1}{|E|-1}$.
As such, the computational complexity of this approach is sensitive to the number of edges in the metric graph, though for every fixed metric graph, the complexity is polynomial.

(We note that, while \eqref{eq.code bound size} can be improved with a more careful analysis, such an improvement is unnecessary in practice.
Instead, one may run the linear program for all $\binom{n+|E|-1}{|E|-1}$ of the assignments $\{t_e\}_{e\in E}$, and if the largest $\delta$ obtained is at most $\min_{e\in E}w_e$, then the relaxation is tight.)

We apply the linear program to test Conjecture~\ref{conj.2} in the context of unit-distance metric trees.
In particular, Theorem~\ref{thm.unique codes in unit-distance metric graphs}(a) reports that $\{C'_k\}_{k=1}^\infty$ is a unicorn sequence, and so for Conjecture~\ref{conj.2}, we are interested in characterizing the optimal codes of size $|V|+k|E|-1=(k+1)|E|$.
We considered all trees with non-trivial automorphism group on up to $7$ vertices, and in each case, we characterized the optimal codes of size $(k+1)|E|$ for $k\in\{1,\ldots,|E|\}$.
(The code is included as an ancillary file in the arXiv version of the paper.)
In all cases, for every optimal code $C$ of these sizes, there exists a nontrivial isometry $g$ for which $gC=C$.
This comports with Conjecture~\ref{conj.2}.
See Figure~\ref{fig.metric trees} for an illustration of these results.

\begin{figure}
\begin{center}
\vspace{-0.75in}
\begin{tabular}{ccc}
\hspace{-0.45in}
\begin{tabular}{c}
\includegraphics[scale=0.4,angle=40]{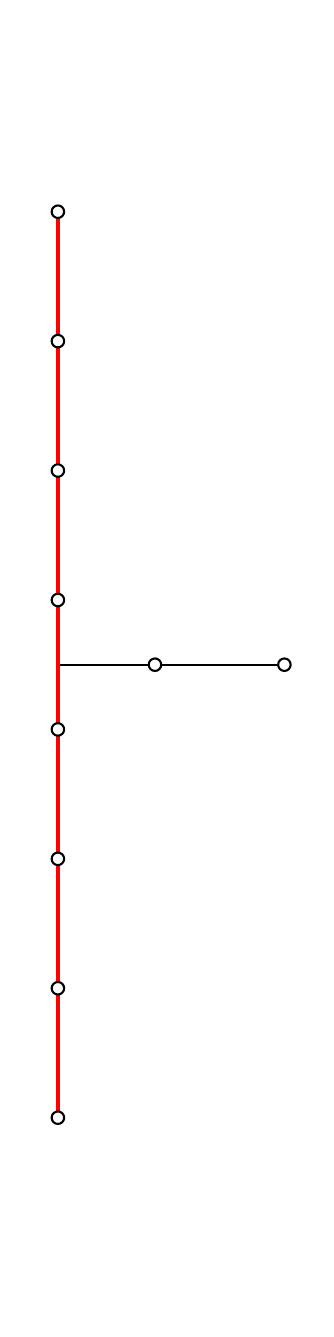}
\end{tabular}
\hspace{-0.5in}
&
\hspace{-0.75in}
\begin{tabular}{c}
\includegraphics[scale=0.4]{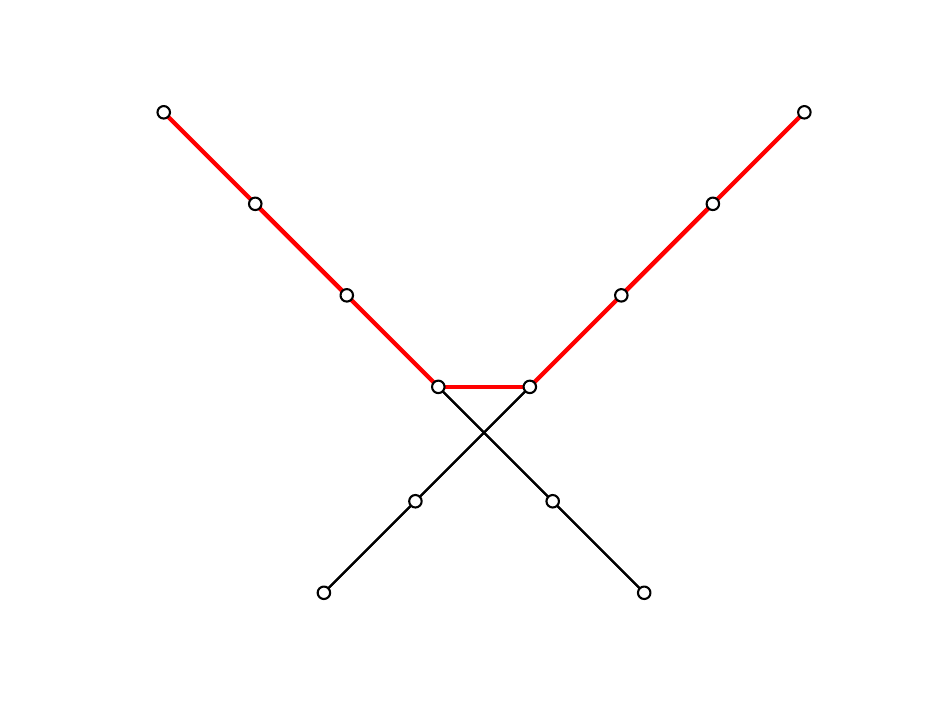}
\end{tabular}
\hspace{-0.5in}
&
\hspace{-1in}
\begin{tabular}{c}
\includegraphics[scale=0.4,angle=135]{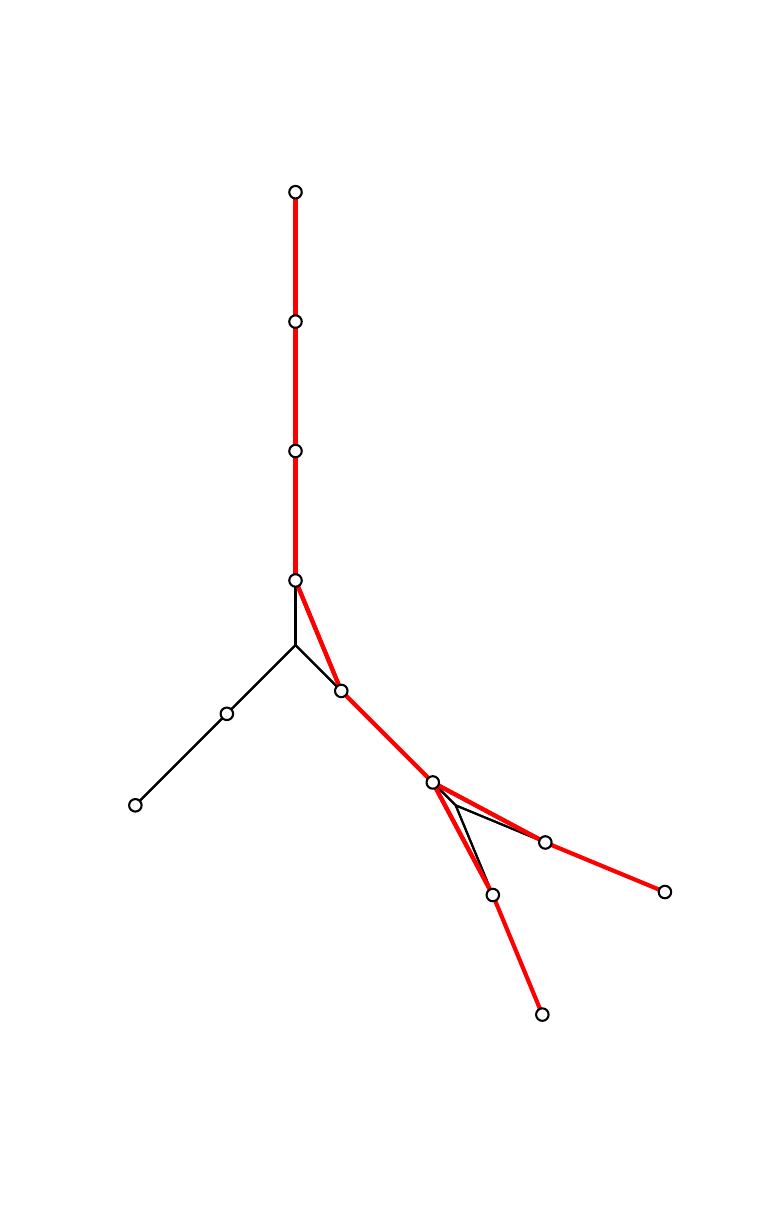}
\end{tabular}
\end{tabular}
\vspace{-0.75in}
\end{center}
\caption{\label{fig.metric trees}
Non-unique optimal codes in metric trees.
For each $m\leq6$, consider all metric trees with nontrivial isometry group consisting of $m$ unit-length edges, and apply the linear programming bound in Subsection~\ref{sec.lp} to find all optimal codes of size $2m$.
In all but the three cases illustrated above, the optimal code is unique (and therefore invariant under the full isometry group of the metric tree).
We draw a red edge between code points whose distance equals the minimum distance.
In each case, the code is invariant under a nontrivial isometry of the metric tree.
}
\end{figure}

\subsection{Ultrametric spaces}

A metric space $(M,d)$ is an ultrametric space if it satisfies the strong triangle inequality
\[
d(x,z)
\leq \max\Big( d(x,y), d(y,z) \Big)
\]
for every $x,y,z\in M$.
This section is concerned with compact ultrametric spaces so that optimal codes are guaranteed to exist.
As a simple example, we may define a metric on $M:=\{0,1\}^\infty$ by $d(x,y)=2^{-k(x,y)}$, where $k(x,y)$ is the smallest $k$ for which $x_k\neq y_k$.
This choice of $(M,d)$ is isometric to the $2$-adic integers, and in general, the $p$-adic integers also form a compact ultrametric space.
The following standard lemma makes it easy to find optimal codes in compact ultrametric spaces:

\begin{lemma}
\label{lem.ultrametric balls}
Let $(M,d)$ be an ultrametric space and select $r>0$. Then for every $x,y\in M$, the open balls $B(x,r)$, $B(y,r)$ are either equal or disjoint.
\end{lemma}

\begin{proof}
Suppose these balls intersect, and select $z\in B(x,r)\cap B(y,r)$.
Then $d(x,y) \leq \max(d(x,z),d(y,z)) < r$, and so $y\in B(x,r)$.
Now select any $u\in B(y,r)$.
Then $d(x,u) \leq \max(d(x,y),d(y,u)) < r$, meaning $u \in B(x,r)$.
Overall, $B(y,r)\subseteq B(x,r)$, and a similar argument gives the reverse containment. 
\end{proof}

Overall, $M$ is partitioned by the open balls of radius $r$.
Let $\Pi(r)$ denote this partitition.
When $M$ is compact, this open cover has a finite subcover, but any proper subcollection would fail to cover, so the full cover $\Pi(r)$ must be finite.
Decreasing $r>0$ has the effect of monotonically refining $\Pi(r)$.
These facts lead to the following:

\begin{lemma}
\label{lem.characterize ultrametric codes}
The optimal codes of size $n$ in the compact ultrametric space $(M,d)$ are the sets that consist of $n$ representatives from $\Pi(\delta)$, where $\delta=\max\{r:|\Pi(r)|\geq n\}$.
\end{lemma}

\begin{proof}
Fix $C\subseteq M$.
For each $x,y\in C$, we have $x\not\in B(y,\delta(C))$, and so by Lemma~\ref{lem.ultrametric balls}, it holds that $B(x,\delta(C))$ is disjoint from $B(y,\delta(C))$.
Since $x$ and $y$ were chosen arbitrarily, it follows that the open balls of radius $\delta(C)$ about each point in $C$ are disjoint, and so $C$ consists of representatives from $n$ of the parts in $\Pi(\delta(C))$.
The result follows.
\end{proof}

Returning to our simple example $M=\{0,1\}^\infty$, notice that $C=\{(w,0_\infty):w\in\{0,1\}^m\}$ is an optimal code by Lemma~\ref{lem.characterize ultrametric codes}.
In particular, $C$ consists of representatives of all $2^m$ parts in $\Pi(\delta)$, where $\delta=2^{-(m+1)}$, and any larger $\delta$ would produce a strictly refined partition.
It turns out that this code is unique up to isometry, and this can be viewed as an instance of a more general result that we prove below.
First, we say that subsets $A$ and $B$ of a metric space $(M,d)$ are \textbf{isometric} if there exists a bijection $f\colon A\to B$ such that $d(f(x),f(y))=d(x,y)$ for every $x,y\in A$.
Next, we say a metric space is \textbf{homogeneous} if isometry between finite subsets extends to an isometry of the entire space.
Also, given a metric space $(M,d)$, we say that $x\in M$ is an \textbf{isolated point} if there exists $\epsilon>0$ such that $d(x,y)<\epsilon$ implies $y=x$.

\begin{lemma}
Let $(M,d)$ be a homogeneous compact ultrametric space with no isolated points.
\begin{itemize}
\item[(a)]
For every $r>0$, the optimal code of size $|\Pi(r)|$ in $(M,d)$ is unique up to isometry.
\item[(b)]
The symmetry strength of $(M,d)$ is infinite.
\end{itemize}
\end{lemma}

\begin{proof}
(a)
By Lemma~\ref{lem.characterize ultrametric codes}, every such optimal code $C$ consists of representatives of $\Pi(r)$.
For any $x,y\in C$, notice that $d(x,y)$ equals the infimum of $s>0$ for which $x$ and $y$ are assigned to the same part in $\Pi(s)$.
For any $A,B\in\Pi(r)$ and any $x,x'\in A$ and $y,y'\in B$, Lemma~\ref{lem.ultrametric balls} then gives that $d(x,y)=d(x',y')$.
It follows that any alternative choice $C'$ of representatives of $\Pi(r)$ is isometric to $C$.
As such, homogeneity implies that $C$ is unique up to isometry.

(b)
Given any finite $C\subseteq M$, select $x\in C$.
Since $(M,d)$ has no isolated points, there exists $\epsilon<\delta(C)$ such that the open ball $B\in\Pi(\delta(C))$ that contains $x$ is refined in $\Pi(\epsilon)$.
Select the largest $\epsilon_1>0$ for which $\Pi(\epsilon_1)$ partitions $B$ into at least two subsets, including $B_0$ and $B_1$, where $x\in B_0$.
Next, select the largest $\epsilon_2>0$ for which $\Pi(\epsilon_2)$ partitions $B_1$ into at least two subsets, including $B_{10}$ and $B_{11}$.
Select $y,y'\in B_{10}$ and $z,z'\in B_{11}$ such that $y\neq y'$ and $z\neq z'$.
(All of these choices are possible since $(M,d)$ has no isolated points.)
Then $C\cup\{y,z\}$ is isometric to $C\cup\{y',z'\}$, and by homogeneity, there must be an isometry $\varphi$ of $(M,d)$ that sends $C\cup\{y,z\}$ to $C\cup\{y',z'\}$.
This isometry is nontrivial since $C\cup\{y,z\}\neq C\cup\{y',z'\}$.
In addition, it holds that $d(y,z)=d(y',z')=\epsilon_2$, 
while all other distances are at least $\epsilon_1>\epsilon_2$, and so $C$ is invariant to $\varphi$.
Finally, the size of $C$ was arbitrary, and so it follows that $(M,d)$ has infinite symmetry strength.
\end{proof}

Notice that Conjecture~\ref{conj.1} trivially holds for unicorn spaces with infinite symmetry strength, while Conjecture~\ref{conj.2} vacuously holds.
It is unclear whether \textit{every} unicorn ultrametric space has infinite symmetry strength.

\subsection{An example in $\ell^p$}

In this section, it is convenient to index the standard basis of $\ell^p$ as $\{e_{(k,\epsilon)}\}$ with $k\in\mathbb{N}$ and $\epsilon\in\{\pm1\}$.
Explicitly, $e_{(k,+1)}=e_{2k-1}$ and $e_{(k,-1)}=e_{2k}$.
Select any nonempty $N\subseteq\mathbb{N}$, and consider the set
\begin{equation}
\label{eq.lpex}
M
:=\{0\}\cup\{\tfrac{1}{k}e_{(k,\pm1)}:k\in N\}\cup\{\tfrac{1}{k}e_{(k,+1)}:k\in \mathbb{N}\setminus N\}
\subseteq\ell^p.
\end{equation}
Observe that every sequence in $M$ either has a constant subsequence or converges to $0$.
Thus, the metric space $(M,\|\cdot\|_p)$ is sequentially compact and therefore compact.

It will be convenient to access the $k$-coordinate of an arbitrary member of $M$, so we define $\pi\colon M\to \mathbb{N}\cup\{0\}$ by $\pi\colon \tfrac{1}{k}e_{(k,\epsilon)}\mapsto k$ and $\pi\colon 0\mapsto 0$.
In addition, consider the involution $\iota\colon M\to M$ defined by $\iota\colon\tfrac{1}{k}e_{(k,\epsilon)}\mapsto\tfrac{1}{k}e_{(k,-\epsilon)}$ for each $k\in N$ and $\epsilon\in\{\pm1\}$ and otherwise $\iota\colon x\mapsto x$.
Then for every subset $S\subseteq N$, the map $g_S\colon M\to M$ defined by
\[
g_S(x)=\left\{\begin{array}{cl}
\iota(x)&\text{if }\pi(x)\in S\\
x&\text{otherwise}
\end{array}\right.
\]
is an isometry of $M$.

\begin{lemma}
Select $N\subseteq\mathbb{N}$, take $M$ defined by \eqref{eq.lpex}, and pick $n\in\mathbb{N}$ with $n>1$.
Let $C\subseteq M$ consist of $n$ points in $M$ of largest $\ell^p$ norm.
Then $C$ is an optimal code in $M$ and is unique up to isometry.
\end{lemma}

\begin{proof}
First, the codes of the described form are unique up to isometry, and so it suffices to show that all other codes are suboptimal.
Suppose $C$ is not of the described form and select $x\in C$ of minimum norm.
Then there exists $y\in M\setminus C$ such that $\|y\|_p>\|x\|_p$.
We claim that $C':=(C\setminus\{x\})\cup \{y\}$ has strictly larger minimum distance.
Select $z\in C\setminus\{x\}$ of minimum norm, and consider any $u,v\in C'$ with $u\neq v$.
Since $u$ and $v$ have disjoint support (and similarly for $x$ and $z$), it holds that
\[
\|u-v\|_p^p
=\|u\|_p^p+\|v\|_p^p
>\|z\|_p^p+\|x\|_p^p
=\|z-x\|_p^p.
\]
It follows that
\[
\delta(C)
=\min_{u\in C\setminus \{x\}}\|u-x\|_p
<\min\Big(\delta(C\setminus\{x\}),\min_{u\in C\setminus \{x\}}\|u-y\|_p\Big)
=\delta(C').
\qedhere
\]
\end{proof}

Overall, $M$ is a unicorn space.
To evaluate Conjectures~\ref{conj.1} and~\ref{conj.2}, there are two cases to consider.
If $N$ is finite, then all optimal codes of size $n>2\cdot\max N$ are invariant under the isometry $g_N$, and so Conjecture~\ref{conj.1} holds.
Since our unicorn sequence consists of codes of all sizes, Conjecture~\ref{conj.2} trivially holds in this case.
On the other hand, if $N$ is infinite, then the symmetry strength of $M$ is infinite.
Indeed, given any finite $C\subseteq M$, select any $m\in N$ with $m>\max\{\pi(x):x\in C\}$.
Then $C$ is invariant under the isometry $g_{\{m\}}$.
As such, Conjecture~\ref{conj.1} trivially holds and Conjecture~\ref{conj.2} vacuously holds in this case.

\subsection{The Hilbert cube}

Consider the Hilbert cube $H$ of sequences $x$ such that $x_k\in[0,\frac{1}{k}]$ for every $k\in\mathbb{N}$.
Notice that $H\subseteq\ell^2$ inherits a metric from the $2$-norm.
In fact, $H$ is compact as a consequence of Tychonoff's theorem, and so we may seek optimal codes in $H$.
For every subset $N\subseteq\mathbb{N}$, the map $g_N\colon H\to H$ defined by
\[
(g_N(x))_k
:=\left\{\begin{array}{cl}
\frac{1}{k}-x_k&\text{if }k\in N\\
x_k&\text{otherwise}
\end{array}\right.
\]
is an isometry of $H$.
We do not know whether $H$ is a unicorn space, but in what follows, we identify small optimal codes in $H$, and in the spirit of Conjectures~\ref{conj.1} and~\ref{conj.2}, each of these are invariant under a nontrivial isometry.

For codes of size $2$, we see that every $x,y\in H$ satisfies
\[
\|x-y\|_2^2
=\sum_{k=1}^\infty (x_k-y_k)^2
\leq \sum_{k=1}^\infty \frac{1}{k^2}
=\frac{\pi^2}{6},
\]
with equality precisely when $|x_k-y_k|=\frac{1}{k}$ for every $k\in\mathbb{N}$.
Up to isometry, this occurs precisely when $x=0$ and $y=\sum_{k=1}^\infty\frac{1}{k}e_k$.
Notice that the uniquely optimal code $\{x,y\}$ is invariant under the isometry $g_\mathbb{N}$.
Codes of size $3$ are less trivial to analyze.

\begin{lemma}
\label{eq.size3hilbert}
Consider $x,y,z\in H$ defined by
\[
x=0,
\qquad
y=\frac{1}{2}e_1+\sum_{k=2}^\infty \frac{1}{k}e_k,
\qquad
z=e_1.
\]
Then $\{x,y,z\}$ is an optimal code in the Hilbert cube and is unique up to isometry.
\end{lemma}

\begin{proof}
First, we show that if $x_1,y_1,z_1\in(0,1)$, then $\{x,y,z\}$ is not optimal.
Without loss of generality, we have $0<x_1\leq y_1\leq z_1<1$.
We may modify $x$ and $z$ in the first coordinate to obtain $x'$ and $z'$ with $x'_1=0$ and $z'_1=1$.
Then $\delta(\{x',y,z'\})>\delta(\{x,y,z\})$, and so $\{x,y,z\}$ is not optimal, as claimed.

Up to isometry, we may therefore assume $0=x_1\leq y_1\leq z_1\leq 1$.
Write 
\[
x=x_\perp,
\qquad
y=y_1e_1+y_\perp,
\qquad
z=z_1e_1+z_\perp.
\]
Then we have
\begin{align}
\min\Big(\|y-x\|_2^2,\|y-z\|_2^2\Big)
\label{eq.smaller_of_smalls}
&=\min\Big(y_1^2+\|y_\perp-x_\perp\|_2^2,(y_1-z_1)^2+\|y_\perp-z_\perp\|_2^2\Big)\\
\label{eq.max_norms}
&\leq\min(y_1^2,(y_1-z_1)^2)+\tfrac{\pi^2}{6}-1\\
\label{eq.y1_in_between}
&\leq\tfrac{z_1^2}{4}+\tfrac{\pi^2}{6}-1.
\end{align}
We combine this with the fact that $t\mapsto\tfrac{t^2}{4}+\tfrac{\pi^2}{6}-1$ and $t\mapsto t^2+\|x_\perp-z_\perp\|_2^2$ are both increasing functions over $t\in[0,1]$ to get
\begin{align}
\delta(\{x,y,z\})^2
\nonumber
&\leq\min\Big(\tfrac{z_1^2}{4}+\tfrac{\pi^2}{6}-1,z_1^2+\|x_\perp-z_\perp\|_2^2\Big)\\
\label{eq.z1to1}
&\leq\min\Big(\tfrac{\pi^2}{6}-\tfrac{3}{4},1+\|x_\perp-z_\perp\|_2^2\Big)\\
\nonumber
&=\tfrac{\pi^2}{6}-\tfrac{3}{4}.
\end{align}
(The last equality uses the fact that $\tfrac{\pi^2}{6}-\tfrac{3}{4}<1$.)
The inequality \eqref{eq.z1to1} is saturated precisely when $z_1=1$, in which case $\delta(\{x,y,z\})^2$ equals \eqref{eq.smaller_of_smalls}.
Next, \eqref{eq.y1_in_between} is saturated precisely when $y_1=z_1/2$.
Finally, when $(y_1,z_1)=(1/2,1)$, the inequality \eqref{eq.max_norms} is saturated precisely when
\[
\|y_\perp-x_\perp\|_2^2
=\|y_\perp-z_\perp\|_2^2
=\tfrac{\pi^2}{6}-1,
\]
which determines $x_\perp$, $y_\perp$ and $z_\perp$ up to isometry.
\end{proof}

The code in Lemma~\ref{eq.size3hilbert} forms the vertices of an obtuse isosceles triangle; one of the squared edge lengths is $1$ and the other two are $\tfrac{\pi^2}{6}-\tfrac{3}{4}$.
It is invariant under the isometry $g_{\{1\}}$.
While we did not determine the optimal codes of size $4$, we ran enough numerical experiments on truncations of the Hilbert cube to formulate the following conjecture:

\begin{conjecture}
\label{conj.hilbert4code}
Put $\alpha:=\frac{\sqrt{2}\pi}{3}$, select $N\subseteq\mathbb{N}$ such that $\sum_{k\in N}\frac{1}{k^2}=\alpha-1$, and define
\[
s:=\sum_{k=2}^\infty\frac{1}{k}e_k,
\qquad
t:=\sum_{k\in N}\frac{1}{k}e_k.
\]
Then $\{0,(1-\frac{1}{2}\alpha)e_1+s,e_1+t,\frac{1}{2}\alpha e_1+s-t\}$ is an optimal code in the Hilbert cube.
\end{conjecture}

This code forms the vertices of a digonal disphenoid; one of the squared edge lengths is $\alpha^2$ and the other five are $\alpha^2-\alpha$.
It is invariant under the isometry $g_{\{1\}\cup N}$.
One may use a greedy approach to construct $N$.
To see this, we first draw inspiration from~\cite{Salzer:47} to prove the following lemma:

\begin{lemma}
\label{lem.salzer}
Given any $x\in(0,\frac{1}{9})$, initialize $x_0:=0$, and for each $k\in\mathbb{N}$, iteratively take $a_k$ to be the smallest integer such that $x_{k-1}+\tfrac{1}{a_k^2}< x$ and then put $x_k:=x_{k-1}+\tfrac{1}{a_k^2}$.
Then $a_{k+1}>a_k$ for every $k\in\mathbb{N}$ and $\sum_{k=1}^\infty\frac{1}{a_k^2}=x$.
\end{lemma}

\begin{proof}
Since $x<\frac{1}{9}$, it holds that $a_1\geq4$.
For each $k\in\mathbb{N}$, our choice of $a_k$ ensures that
\begin{equation}
\label{eq.key estimate for iteration}
\tfrac{1}{a_k^2}
< x-x_{k-1}
\leq\tfrac{1}{(a_k-1)^2}.
\end{equation}
This in turn implies
\[
\tfrac{1}{a_{k+1}^2}
< x-x_k
=x-x_{k-1}-\tfrac{1}{a_k^2}
\leq\tfrac{1}{(a_k-1)^2}-\tfrac{1}{a_k^2}.
\]
Rearranging then gives
\[
a_{k+1}>(\tfrac{1}{(a_k-1)^2}-\tfrac{1}{a_k^2})^{-1/2},
\]
and the right-hand side is at least $a_k$ whenever $a_k\geq 4$.
Overall, we inductively have both $a_{k+1}>a_k$ and $a_k\geq k+3$ for every $k\in\mathbb{N}$, and so \eqref{eq.key estimate for iteration} implies the desired convergence.
\end{proof}

In fact, we can use Lemma~\ref{lem.salzer} to construct \textit{multiple} choices of $N\subseteq\mathbb{N}$ with the property that $\sum_{k\in N}\frac{1}{k^2}=\alpha-1$.
For example, we first take $x:=(\alpha-1)-(\tfrac{1}{2^2}+\tfrac{1}{3^2}+\tfrac{1}{4^2})\in(0,\frac{1}{9})$ and run the iteration in Lemma~\ref{lem.salzer} to get $a_1=5>4$, etc.
This means we can select
\[
N=\{2,3,4,a_1,a_2,\ldots\}.
\]
Alternatively, if we take $x':=(\alpha-1)-(\tfrac{1}{2^2}+\tfrac{1}{3^2}+\tfrac{1}{4^2}+\tfrac{1}{6^2}+\tfrac{1}{7^2})\in(0,\frac{1}{9})$, then the iteration in Lemma~\ref{lem.salzer} gives $a'_1=11>7$, etc.
As such, we can also select
\[
N'=\{2,3,4,6,7,a'_1,a'_2,\ldots\}.
\]
Overall, the codes described in Conjecture~\ref{conj.hilbert4code} exist, but are not unique.

\section{Discussion}

In this paper, we made explicit predictions concerning emergent symmetries in optimal codes.
Many open problems remain, and we collect them in this section.

In the context of the equilateral triangle, our conjectures correspond to conjectures of Erd\H{o}s and Oler~\cite{Oler:61} and Lubachevsky, Graham and Stillinger \cite{LubachevskyGS:97}.

For tori, we conjecture that each of the following are unicorn spaces:
\[
\mathbb{R}^4/D_4,
\qquad
\mathbb{R}^8/E_8,
\qquad
\mathbb{R}^{24}/\Lambda_{24}.
\]
In fact, we believe that each of these spaces admits a unicorn sequence of lattice codes, but a proof of uniqueness is required.
For $\mathbb{R}^2/A_2$, Conjecture~\ref{conj.2} is related to conjectures in \cite{DickensonGKX:11,ConnellyD:14}, of which the smallest open case is the following:
The optimal codes of size $8$ are obtained by removing any point from any lattice code of size $9$.
Presumably, this case can be resolved using techniques from~\cite{MusinN:16}.
We believe similar phenomena hold for $\mathbb{R}^m/L$ with $L\in\{D_4,E_8,\Lambda_{24}\}$.

Metric graphs provide a rich family of compact metric spaces to test our conjectures, and their low dimensionality makes this a feasible endeavor.
Our linear programming bounds provide a systematic approach to hunt for counterexamples to our conjectures.
Presumably, these methods can be used to prove our conjectures in the setting of metric graphs, but it is unclear how to do so.

Is it the case that every unicorn ultrametric space has infinite symmetry strength?
What are the optimal codes of size at least $4$ in the Hilbert cube?
Is the Hilbert cube a unicorn space?
More broadly, what other unicorn spaces are there?

Finally, while we drew inspiration from the orthoplex spherical code, the formalism of our predictions requires sequences of codes that appear to preclude important metric spaces such as the sphere.
We are also interested in real and complex projective spaces, as well as other Grassmannian spaces.
How might one formulate analogous predictions in these settings?

\section*{Acknowledgments}

Much of this work was conducted during the SOFT 2020:\ Summer of Frame Theory virtual workshop. 
CC was partially supported NSF DMS 1555149 and NSF DMS 1839918.
DGM was partially supported by AFOSR FA9550-18-1-0107 and NSF DMS 1829955.
HP was partially supported by an AMS-Simons Travel Grant. The authors would also like to thank the reviewer for their feedback, in particular a suggestion to simplify Case II in the proof of Theorem~\ref{thm.unique codes in unit-distance metric graphs}(b).

\end{document}